\documentclass[11pt]{amsart}
\usepackage{amssymb}
\usepackage{graphicx}
\usepackage{url}
\usepackage{enumerate}
\usepackage[latin1]{inputenc}
\usepackage{color}
\usepackage{leftidx}
\usepackage{tikz}
\usepackage{tikz-cd}

\newcommand\restr[2]{{
  \left.\kern-\nulldelimiterspace 
  #1 
  \vphantom{\big|} 
  \right|_{#2} 
  }}

\newtheorem{theorem}{Theorem}

\newtheorem{lemma}[theorem]{Lemma}
\newtheorem{corollary}[theorem]{Corollary}

\theoremstyle{definition}
\newtheorem{definition}[theorem]{Definition}
\newtheorem{remark}[theorem]{Remark}
\newtheorem{example}[theorem]{Example}
\newtheorem{claim}[theorem]{Claim}

\begin{document}

\title[Construct $b$-symplectic toric manifolds from toric manifolds]{Construct $b$-symplectic toric manifolds from toric manifolds}

\author{Mingyang Li}\address{School of Mathematics, University of Science and Technology of China, Hefei, Anhui, China.}\email{lmy0312@mail.ustc.edu.cn}
\date{\today}

\begin{abstract}
In \cite{btoric}, Guillemin et al. proved a Delzant-type theorem
which classifies $b$-symplectic toric manifolds. More generally, in
\cite{torus} they proved a similar convexity result for general
Hamiltonian torus action on $b$-symplectic manifolds. In this paper
we provide a new way to construct $b$-symplectic toric manifolds
from usual toric manifolds. Conversely, through this way we can also
decompose a $b$-symplectic toric manifolds to usual toric manifolds.
Finally we will try to prove that this kind of decomposition is
useful.
\end{abstract}

\maketitle

\tableofcontents

\section{Introduction}

It is a very famous and beautiful result that the image of the
moment map of a compact symplectic toric manifold is a rational
convex polytope, which we call Delzant polytope. Moreover, the image
of the moment map classifies compact symplectic toric manifolds up
to equivariant diffeomorphisms. See \cite{guilleminsternberg},
\cite{guilleminsternberg2} and \cite{delzant}. More recently, in
\cite{btoric} Guillemin et al. generalize this result to
$b$-symplectic toric manifold, which is a symplectic manifold except
a hypersurface and admits a Hamiltonian toric action with respect to
this singular symplectic form. In this generalized $b$-symplectic
toric case, because of the fact that the symplectic form is singular
along some hypersufaces, the definition of moment map is also
adjusted a little, and the image of the new moment map will be so
called $b$-Delzant polytope. In \cite{btoric}, they constructed
corresponding $b$-symplectic toric manifold from $b$-Delzant
polytope by using symplectic cutting.

In this paper, we will give a way to construct the $b$-symplectic
toric manifold by cutting and gluing several usual symplectic toric
manifolds which belong to some special class of symplectic toric
manifolds. The gluing process will be done along two hypersurfaces,
and such a construction will tell us more topology information about
the $b$-symplectic manifolds. Conversely, using this kind of idea we
can also decompose every $b$-symplectic manifolds to some special
class of usual toric manifolds, which are easily characterized
combinatorially.

This paper is organized as follows: In Section \ref{section prelim}
we will go over some basic definitions and properties of
$b$-symplectic manifolds, in Section \ref{section btoric} we will
review the $b$-Delzant type theorem proved in \cite{btoric}. In
Sections \ref{section construction} and \ref{section decomposition}
we will construct $b$-symplectic toric manifolds from some special
class of toric manifolds, and converse the construction to decompose
$b$-symplectic toric manifolds. And finally in Section \ref{section
application}, we will try to use this kind of construction to prove
some properties of the $b$-symplectic toric manifolds.
\section{Preliminaries}\label{section prelim}

In this part, we will go over some basic definitions and properties
of $b$-symplectic manifolds. We will also give some basic examples.
More details can be found in \cite{scott} and \cite{guimipi12}. As
for basic Poisson geometry, the reader can refer to \cite{poisson}.

\subsection{$b$-manifolds, $b$-tangent and $b$-cotangent bundles}

When we are talking about $b$-manifold, we mean an oriented smooth
manifold $M$ with a closed smooth hypersurface $Z$.

\begin{definition}
A \textbf{$b$-manifold} is a pair $(M,Z)$, where $M$ is an oriented
smooth manifold $M$ and $Z$ is a closed smooth hypersurface in $M$.
We call the hypersurface $Z$ \textbf{singular hypersurface} (It may not be
connected). A \textbf{$b$-map} is a smooth orientation preserving
map
\begin{align*}
f:(M,Z)\longrightarrow(M',Z')
\end{align*}
such that $f:M\rightarrow M'$ is smooth orientation preserving,
$f^{-1}(Z')=Z$ and $f$ is transverse to $Z'$.
\end{definition}

The essential difference between $b$-manifold and a usual manifold
is the definition of tangent bundle. Assume $U$ is an open subset of
the manifold $U\subset M$, we use $\Gamma(U,\leftidx{^b}TM)$ to
denote the vector space consisting of vector fields on $U$, which
are tangent to $Z$ at points of $Z$.

\begin{definition}
We define the \textbf{$b$-tangent bundle} $\leftidx{^b}TM$ of a
$b$-manifold $(M,Z)$ to be the vector bundle defined by the locally
free sheaf $\Gamma(U,\leftidx{^b}TM)$. The $b$-cotangent bundle
$\leftidx{^b}T^*M$ is defined to be the dual bundle of
$\leftidx{^b}TM$.
\end{definition}
By the above definition, if we are at a point $p\in M \setminus Z$,
then the $b$-tangent space $\leftidx{^b}T_pM\simeq T_pM\simeq
\text{span}\{\frac{\partial}{\partial
x_1},\cdots,\frac{\partial}{\partial x_n}\}$, where
$\{x_1,\cdots,x_n\}$ is a local coordinate system. But if we are at
a point $p\in Z$, then $\leftidx{^b}T_pM\simeq
\text{span}\{f\frac{\partial}{\partial
f},\cdots,\frac{\partial}{\partial x_n}\}$. Here $f$ is a defining
function for the hypersurface $Z$ (we need to require $df\neq0$ at
points of $Z$) and $\{f,x_2,\cdots,x_n\}$ is a local coordinate
system around p.

Similarly we can define $b$-differential $k$-forms to be sections of
$\Lambda^k(\leftidx{^b}T^*M)$. We denote the space of
$b$-differential $k$-forms as $\leftidx{^b}\Omega^k(M)$. In general,
if we have a defining function $f$ for the hypersurface $Z$, then
every $b$-differential $k$-form can be written as
\begin{align*}
w=\alpha\wedge\frac{df}{f}+\beta \text{,\ }
\alpha\in\Omega^{k-1}(M)\text{\ and\ }\beta\in\Omega^k(M),
\end{align*}
where $\alpha$ and $\beta$ may not be unique. And using this we can
extend the exterior differential operator $d$ to
$\leftidx{^b}\Omega^*(M)$ by defining
$$dw=da\wedge\frac{df}{f}+d\beta.$$
This helps us to form a chain complex
$$0\rightarrow\leftidx{^b}\Omega^0(M)\rightarrow\leftidx{^b}\Omega^1(M)\rightarrow\cdots\rightarrow\leftidx{^b}\Omega^n(M).$$
The above \textbf{$b$-de Rham complex} will give us
\textbf{$b$-cohomology group} $\leftidx{^b}H^*(M)$. For more
details, one can find in \cite{guimipi12}.

The $b$-forms may blow up at points of $Z$, but we can still
integrate compactly supported $b$-forms over $M$: given a $b$-form
$\eta \in \leftidx{^b}{\Omega^n(M)}$, we define the integral of
$\eta$ over $M$ to be the singular integral
$$\leftidx{^b}\int_M\eta=\lim_{\epsilon\rightarrow 0}\int_{|f|>\epsilon}\eta,$$
where $f$ is a defining function of the hypersurface $Z$. This
definition does not depend on the choice of defining functions. See
\cite{scott}.

In this $b$-manifold case, if we have a local defining function $f$
for the hypersurface $Z$, then $\frac{df}{f}$ will be a $b$-form
which is closed but not exact. To make it exact we have to enlarge
the space of smooth functions on $b$-manifold $M$ to include
functions such as $\log|f|$. So we have the following definition.
\begin{definition}
Given a $b$-manifold $(M,Z)$, we define the sheaf of
\textbf{$b$-functions} on $M$, $\leftidx{^b}{\mathcal{C}}^{\infty}$
to be
\begin{align*}
\leftidx{^b}{\mathcal{C}}^{\infty}(U)=c\log|y|+f,
\end{align*}
where $f\in\mathcal{C}^\infty(U)$, $c\in\mathbb{R}$ and $y$ is a
defining function for $U\cap Z$.
\end{definition}
Such an enlargement will also help us to define $b$-moment map in
the future.

\subsection{$b$-symplectic manifolds}

In this subsection we will discuss $b$-symplectic manifolds.
\begin{definition}
Assume $(M,Z)$ is a $2n$ dimensional $b$-manifold. We say
$\omega\in\leftidx{^b}{\Omega}^2(M)$ is $b$-symplectic if it is a
closed $b$-form such that
$$\omega^n_p\neq0$$
for each point $p$, as an element in $\leftidx{^b}{T_p^{2n}M}$. When
a $b$-manifold $(M,Z)$ admits a $b$-symplectic form, we say it is a
$b$-symplectic manifold.
\end{definition}

Here is a very simple example.
\begin{example}
Consider the Euclidean space $\mathbb{R}^{2n}$ with coordinates
$\{x_1,y_1,\cdots,x_n,y_n\}$. Assume the hypersurface is
$\{x_1=0\}$. Then we have a $b$-symplectic form $\omega$,
$$\omega=\frac{dx_1}{x_1}\wedge dy_1+\cdots+dx_n\wedge dy_n.$$
\end{example}

\begin{definition}
A $b$-map $f$ between two $b$-symplectic manifolds $(M,Z)$ and
$(M',Z')$ is a $b$-symplectomorphism, if $f^*\omega'=\omega$, where
$\omega$ and $\omega'$ is the $b$-symplectic form of $M$ and $M'$
respectively.
\end{definition}

From the perspective of Poisson geometry we can also study
$b$-symplectic manifolds. Recall that there is a one-to-one
correspondence between Poisson structures on a manifold and
bi-vector fields $\pi\in\chi^2(M)$ (where by $\chi (M)$ we mean
vector fields on $M$), which satisfy the condition $[\pi,\pi]=0$ on
a manifold. If we view a bi-vector $\pi\in\chi^2(M)$ as a map
$${\pi}^{\#}:\Omega^1(M)\longrightarrow\chi(M)$$
which is defined by contraction, then it will be a bundle map
between $T^*M$ and $TM$. If it is an isomorphism we can consider its
inverse $\omega:TM\rightarrow T^*M$. And this gives us a symplectic
form. Conversely given a symplectic structure we can also construct
a Poisson structure just as above (We will get a non-degenerate
Poisson structure). In our case, if we do the construction for a
$b$-symplectic form we in fact get a Poisson bi-vector $\pi$ whose
top exterior product as a section of $\Lambda^{2n}(TM)$ vanishes
transversely on the hypersurface $Z$, and is nonzero at other points
(One can see this easily by Weinstein's result on local structure of
a Poisson manifold. See \cite{weinstein} and \cite{poisson}.). So
$b$-symplectic manifolds are in fact a special kind of Poisson
manifolds.

\subsection{Torus Hamiltonian actions on $b$-symplectic manifolds}
Let $\mathbb{T}^n$ be a $n$-dimensional torus and acts on a manifold
$M$. We use $\mathfrak{t}$ to denote its Lie algebra and
$\mathfrak{t}^*$ to denote the dual of its Lie algebra. In the
symplectic case we call it a \textbf{Hamiltonian action} if there
exists a equivariant \textbf{moment map} $\mu:M\rightarrow
\mathfrak{t}^*$ such that for every $X\in\mathfrak{t}$ we have,
$$d\mu^X=-\iota_{X^{\#}}\omega.$$
Here we use $\mu^X$ to denote the function
$\mu^X=\langle\mu,X\rangle$, and $X^{\#}$ is the vector field on $M$
corresponding to the infinitesimal action by $X$,
$$X^{\#}(p)=\frac{d}{dt}(\exp(tX)\cdot p).$$

In the setting of $b$-symplectic manifold, we can define Hamiltonian
action similarly.
\begin{definition}
A $\mathbb{T}^n$ action on a $b$-symplectic manifold $(M,Z,\omega)$
will be called \textbf{Hamiltonian} if for any $X$,
$Y\in\mathfrak{t}$, we have,
\begin{itemize}
\item the one form $\iota_{X^{\#}}\omega$ is $b$-exact, i.e. exists a primitive $H_X\in\leftidx{^b}{\mathcal{C}}^\infty(M)$;
\item $\omega(X^{\#},Y^{\#})=0$.
\end{itemize}
Such a Hamiltonian action will be called \textbf{toric} if it is an
effective action and dim($\mathbb{T}^n$)=$\frac{1}{2}$dim$(M)$.
\end{definition}

\begin{example}
We can consider the following simple example. Let
$(\mathbb{S}^2,Z,\omega)$ be the $2$-dimensional sphere with
standard $b$-symplectic form $\omega=\frac{dh}{h}\wedge d\theta$,
with $h\in[-1,1]$, $\theta\in[0,2\pi]$, and the corresponding
hypersurface is $Z=\{h=0\}$. The $\mathbb{S}^1$ action on the sphere
will be rotation with respect to the axis as in the picture below.
In another word, will be the flow of the vector field
$\frac{\partial}{\partial\theta}$. If we consider the function
$\mu(h,\theta)=\log|h|$, then we will have,
$$d\mu^X=-\iota_{X^{\#}}\omega.$$
Hence the action satisfies the definition above and the action is
actually a Hamiltonian action. Moreover, this is a toric action. The
image of the above function $\log|h|$ is as in the picture.
\end{example}
\begin{figure}[ht]
\includegraphics[width=5.9cm,height=3.5cm]{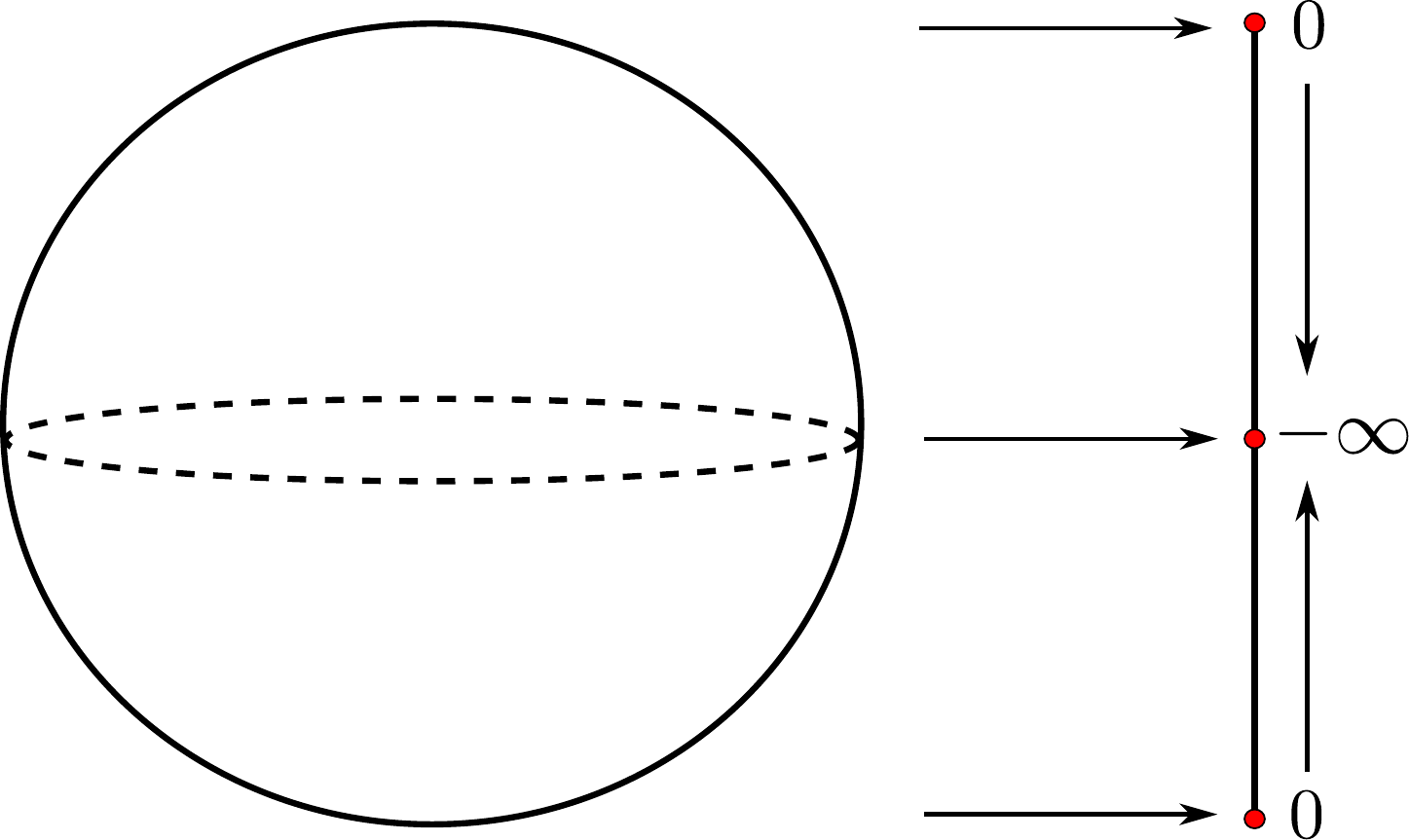}
\caption{The "moment map" of the $\mathbb{S}^1$-actions on $M$.}
\label{fig:sphere}
\end{figure}
Moreover, if we consider
$(\mathbb{T}^2,Z=\{\theta_1\in\{0,\pi\}\},\omega=\frac{d\theta_1}{\sin\theta_1}\wedge
d\theta_2)$, and let the torus $\mathbb{S}^1$ act like rotation on
$\theta_2$ coordinate. Then if we select the "moment map" as
$\log|\frac{1+\cos\theta_1}{\sin\theta_1}|$, the $\mathbb{S}^1$
action will also be a Hamiltonian action. The important point is
that the torus $\mathbb{T}^2$ does not admit a symplectic toric
structure. For details see \cite{btoric}.

\subsection{$b$-moment maps and $b$-Delzant polytopes}

As in \cite{btoric}, they showed that there exists a $b$-moment map
and its image will satisfies some kind of Delzant property. But
because of the fact that our $b$-symplectic form blows up along the
hypersurface $Z$, the image of our $b$-moment map cannot fall in one
Euclidean space like $\mathfrak{t}^*$ as before. We conclude the
results and sketch the proofs that appear in \cite{btoric}.
\begin{definition}
Given a Hamiltonian $\mathbb{T}^k$ action on our $b$-manifold, for
each connected component $Z'$, by the hypothesis that the action is
Hamiltonian, the $b$-form $\iota_{X^{\#}}\omega$ has a primitive in
$\leftidx{^b}{\mathcal{C}}^\infty$. This primitive can be written as
$c\log|y|+g$ in a neighborhood of $Z'$, where $y$ is the local
defining function of $Z'$, $c\in\mathbb{R}$ and $g$ is smooth. The
map $X\mapsto c$ is a well-defined homomorphism, and will correspond
to an element $\nu_{Z'}\in\mathfrak{t}^*$. This is called the
\textbf{modular weight} of $Z'$. The kernel of the modular weight
$\nu_Z'$ will be denoted as $\mathfrak{t}_Z'$.
\end{definition}

In \cite{btoric} they proved that such a modular weight $\nu_Z$ for
toric action is nonzero. Moreover, it was also proved that for toric
action if $Z_1$ and $Z_2$ are two adjacent (i.e. adjacent to the
same connected component of $M\backslash Z$), then the modular
weights of $Z_1$ and $Z_2$ satisfy:
$$\nu_{Z_1}=k\nu_{Z_2}\text{, for some }k<0.$$
That is to say, all the modular weights are in the "same direction",
but for two adjacent modular weights, they will have different sign.
We will see in the future this will help us to glue different
$\mathfrak{t}^*$ together. See claim 20 in \cite{btoric}.

\begin{definition}
The \textbf{adjacency graph} of a $b$-manifold $(M,Z)$ is a graph
$G=(V,E)$. It has a vertex for each component of $M\backslash Z$ and
an edge for each connected component of $Z$.
\end{definition}

\begin{definition}
Given a Hamiltonian $\mathbb{T}^k$ action on a $b$-symplectic
manifold $(M,Z)$, we define the \textbf{weighted adjacency graph} of
the $b$-symplectic manifold with the Hamiltonian action above to be
$\mathcal{G}=(G,w)$, where $G$ is the adjacency graph of $M$ and $w$
is a weight function on the set of edges
$w:E\rightarrow\mathfrak{t}^*$. The weight function $w$ is defined
to be, mapping an edge $e\in E$ to the modular weight of the
corresponding connected component in $Z$.
\end{definition}

In the case of toric action, . To help the reader understand those
concepts, we consider the b-symplectic toric manifolds
$\mathbb{S}^2$ and $\mathbb{T}^2$ mentioned above as examples.

\begin{example}
For the sphere case, there will be two vertexes for the two
components in $M\backslash Z$ and only one edge for the only one
component in $Z$. For the torus case, similarly, there will exist
two vertexes and two edges. But in this case it will form a closed
loop, as shown in the following pictures.
\end{example}

\begin{figure}[ht]
\includegraphics[height=3.5cm]{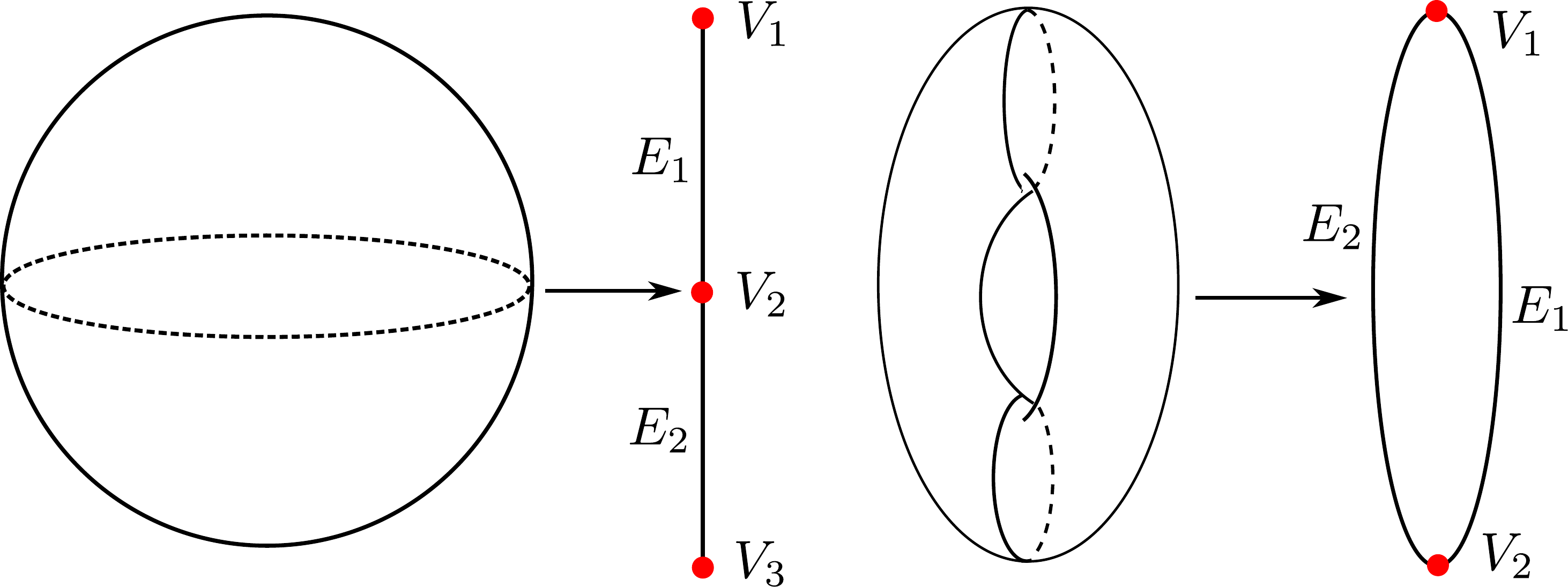}
\caption{The adjacency graph of $\mathbb{S}^2$ and $\mathbb{T}^2$.}
\label{fig:graph1}
\end{figure}

Now, have already defined the weighted adjacency graph of a
$b$-symplectic manifold with a Hamiltonian action, we can now
construct the moment codomain of a $b$-symplectic toric manifold. We
construct the \textbf{$b$-moment codomain}
$(\mathcal{R}_{\mathcal{G}},\mathcal{Z}_{\mathcal{G}},\hat{x})$,
where $(\mathcal{R}_{\mathcal{G}},\mathcal{Z}_{\mathcal{G}})$ is a
$b$-manifold and
$\hat{x}:\mathcal{R}_{\mathcal{G}}\backslash\mathcal{Z}_{\mathcal{G}}\rightarrow\mathfrak{t}^*$
is a smooth map, such that:
\begin{itemize}
\item If $G$ is a single vertex, we define the triple
$(\mathcal{R}_{\mathcal{G}},\mathcal{Z}_{\mathcal{G}},\hat{x})=(\mathfrak{t}^*,\varnothing,id)$;
\item If $G$ has at least two vertices, then for each $e\in E$, we have the
corresponding modular weight $w(e)$. Let
$\mathfrak{t}_w=(w(e))^\perp\subset\mathfrak{t}$. We will use
modular weights to glue different $\mathfrak{t}^*$ together as said
before. As a set, we define:
\begin{align*}
\mathcal{R}_{\mathcal{G}}&=\mathfrak{t}^*\times V
\sqcup\mathfrak{t}_w^*\times E\\
\mathcal{Z}_{\mathcal{G}}&=\mathfrak{t}_w^*\times E.
\end{align*}
and define $\hat{x}((x,v))=x$.
\end{itemize}

$\mathcal{R}_{\mathcal{G}}$ can be endowed with a smooth structure
such that $(\mathcal{R}_{\mathcal{G}},\mathcal{Z}_{\mathcal{G}})$ is
a $b$-manifold, which also makes $\hat{x}$ is a smooth map. For
technical details see \cite{btoric}. And this will be the codomain
of the moment map.

\begin{definition}
Given a $b$-symplectic manifold with a effective Hamiltonian
$\mathbb{T}^k$ action such that the modular weights are all nonzero.
Then a $b$-\textbf{moment map} for an effective Hamiltonian
$\mathbb{T}^k$ action will be a $\mathbb{T}^k$-invariant b-map
$\mu:M\rightarrow \mathcal{R}_{\mathcal{G}}$ such that for any $X\in
\mathfrak{t}$, the function
$\mu^X(p)=\langle\hat{x}\circ\mu(p),X\rangle$ is linear with respect
to $X$, and satisfies:
$$\iota_{X^{\#}}\omega=-d\mu^X$$
\end{definition}

It was proved that for any $b$-symplectic toric manifold there
always exists a $b$-moment map. See theorem 27 in \cite{btoric}.

\begin{example}
We still consider the sphere $\mathbb{S}^2$, but this time we can
consider the case where there are more than one singular
hypersurfaces. The image of the moment map will be as below.
\end{example}

\begin{figure}[ht]
\includegraphics[height=3.5cm]{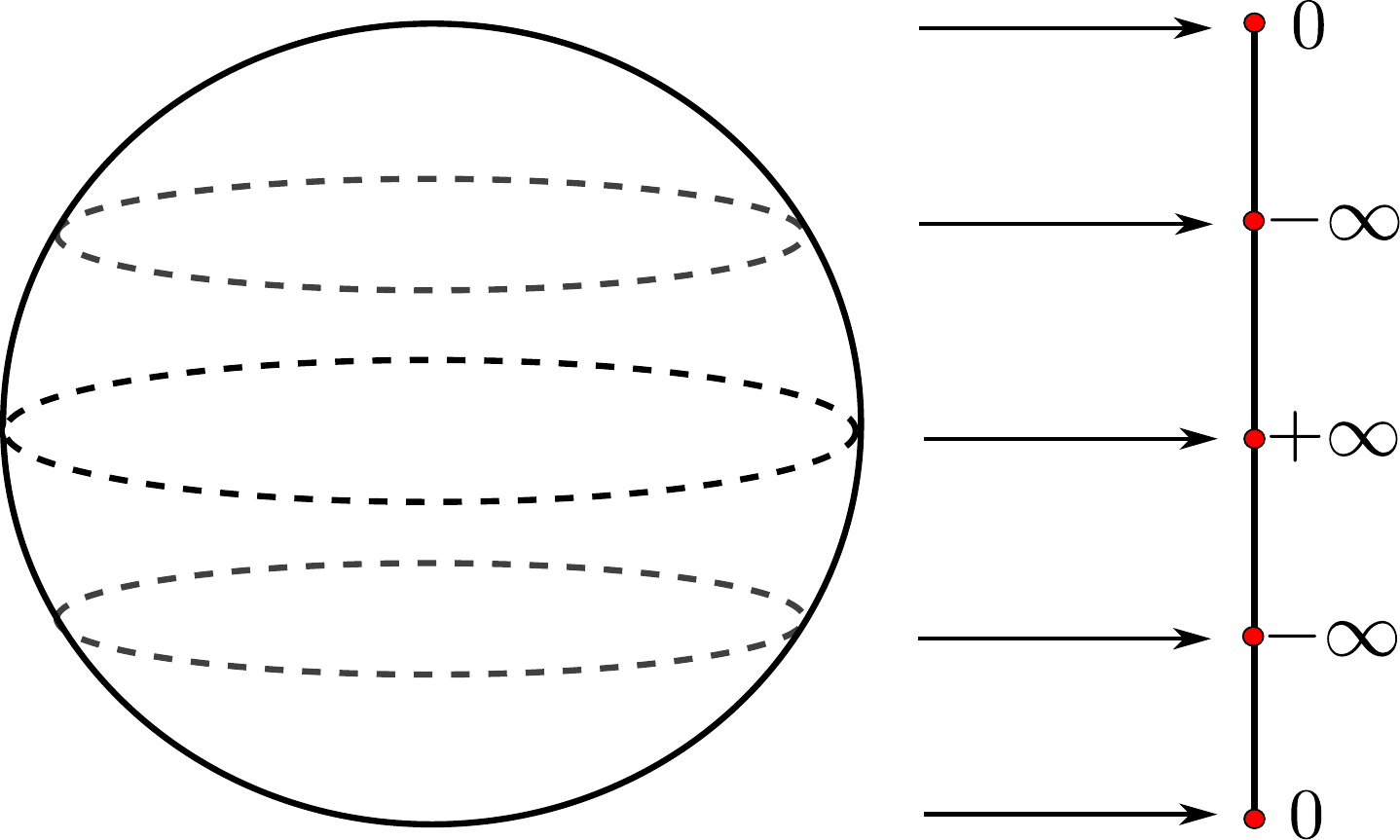}
\caption{The $b$-moment map image of $\mathbb{S}^2$ with many singular hypersurfaces.}
\label{fig:sphere2}
\end{figure}

\section{The $b$-Delzant classification}\label{section btoric}

In this section, we will describe the $b$-Delzant type theorem
proved in \cite{btoric}. We first need to define what is a rational
convex polytope, or a $b$-Delzant polytope in
$\mathcal{R}_{\mathcal{G}}$.

\begin{definition}
We use $A_{X,k,v}$ and $B_{Y,k}$ to denote the following two kinds
of hyperplanes in $\mathcal{R}_{\mathcal{G}}$, where
$X\in\mathfrak{t}$, $Y\in\mathfrak{t}_{w}$, $k\in\mathbb{R}$ and $v$
is a vertex of $G$.
$$A_{X,k,v}=\{(v,\xi)|\langle\xi,X\rangle=k\}\subset{v}\times\mathfrak{t}^*\subset\mathcal{R}_{\mathcal{G}},$$
$$B_{Y,k}=\overline{\{(v,\xi)|\langle\xi,Y\rangle=k\text{, }v\text{ a vertex of }G\}}\subset\overline{\mathcal{R}_{\mathcal{G}}\backslash\mathcal{Z}_{\mathcal{G}}}=\mathcal{R}_{\mathcal{G}}.$$
We call the closure of each component in the complement of the
hyperplane a \textbf{half-space} in $\mathcal{R}_{\mathcal{G}}$.
\end{definition}

The following picture shows some examples of those hyperplanes.
\begin{figure}[ht]
\includegraphics[width=4cm,height=6cm]{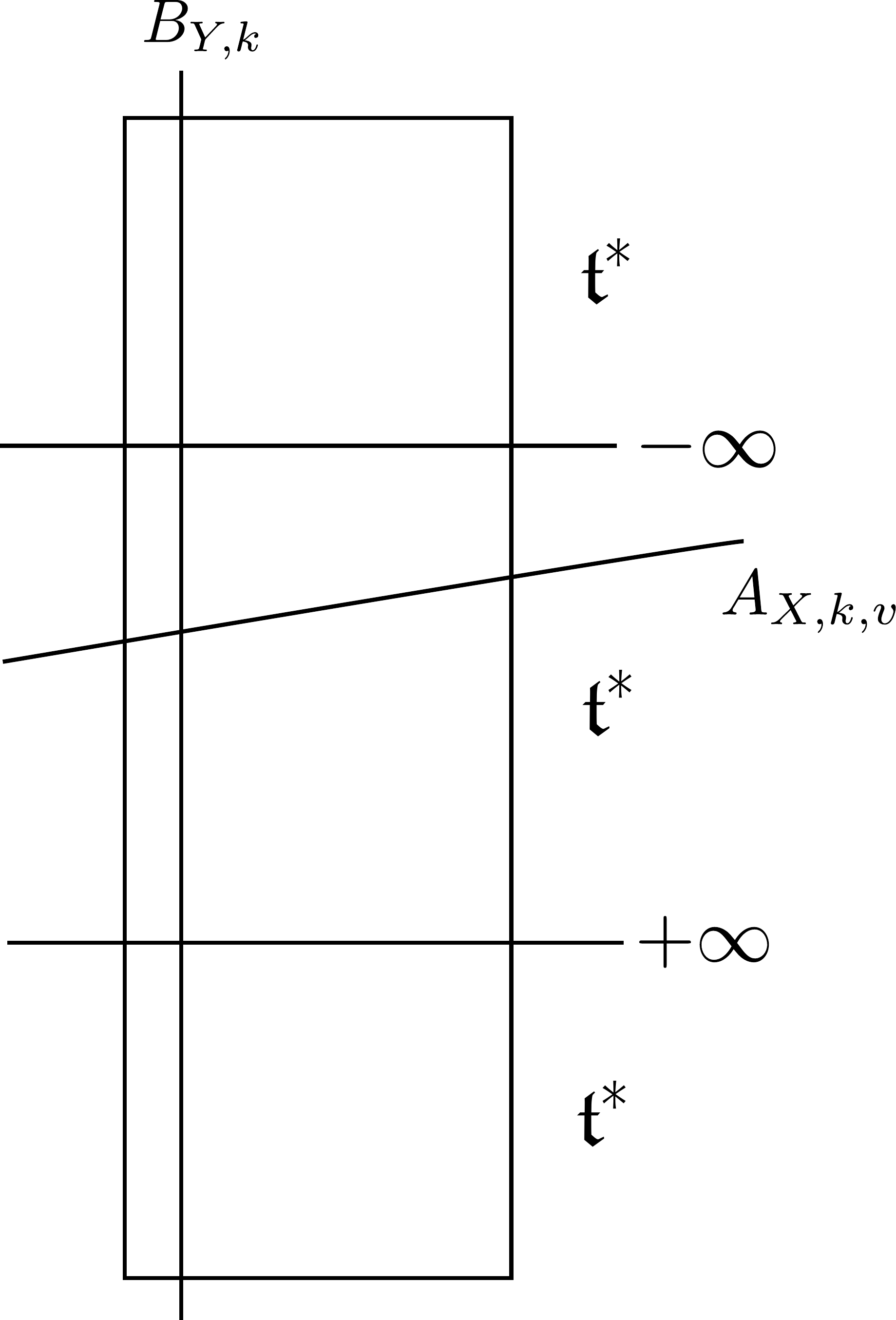}
\caption{Examples of hyperplanes.}
\label{fig:hyperplane}
\end{figure}
\begin{definition}
A $b$-\textbf{polytope} $\Delta$ in $\mathcal{R}_{\mathcal{G}}$ is a
bounded connected set such that each component of
$\Delta\cap\mathcal{R}_{\mathcal{G}}\backslash\mathcal{Z}_{\mathcal{G}}$
is an intersection of half-spaces.
\end{definition}

The definition of $b$-Delzant is similar to the usual definition.
For a $b$-polytope we can talk about vertices, facets and edges as
before.

\begin{definition}
When $G$ is a line, we say a $b$-polytope is $b$-\textbf{Delzant} if
for every vertex $v$ of $\triangle$, there exists a lattice basis
$\{u_i\}$ of $\mathfrak{t}^*$ such that the edges linking to $v$ can
be written as $v+tu_i$ for $t\geqslant0$. When $G$ is a circle, a
$b$-polytope is $b$-\textbf{Delzant} if the polytope
$\Delta_Z=\Delta\cap\mathfrak{t}^*_w$ is Delzant.
\end{definition}

The following picture is an example of $b$-Delzant polytope. In this
picture it shows a $b$-Delzant polytope in a $b$-codomain consisting
of two $\mathfrak{t}^*$ glued together along some modular weight
direction.
\begin{figure}[ht]
\includegraphics[width=3.2cm,height=3.5cm]{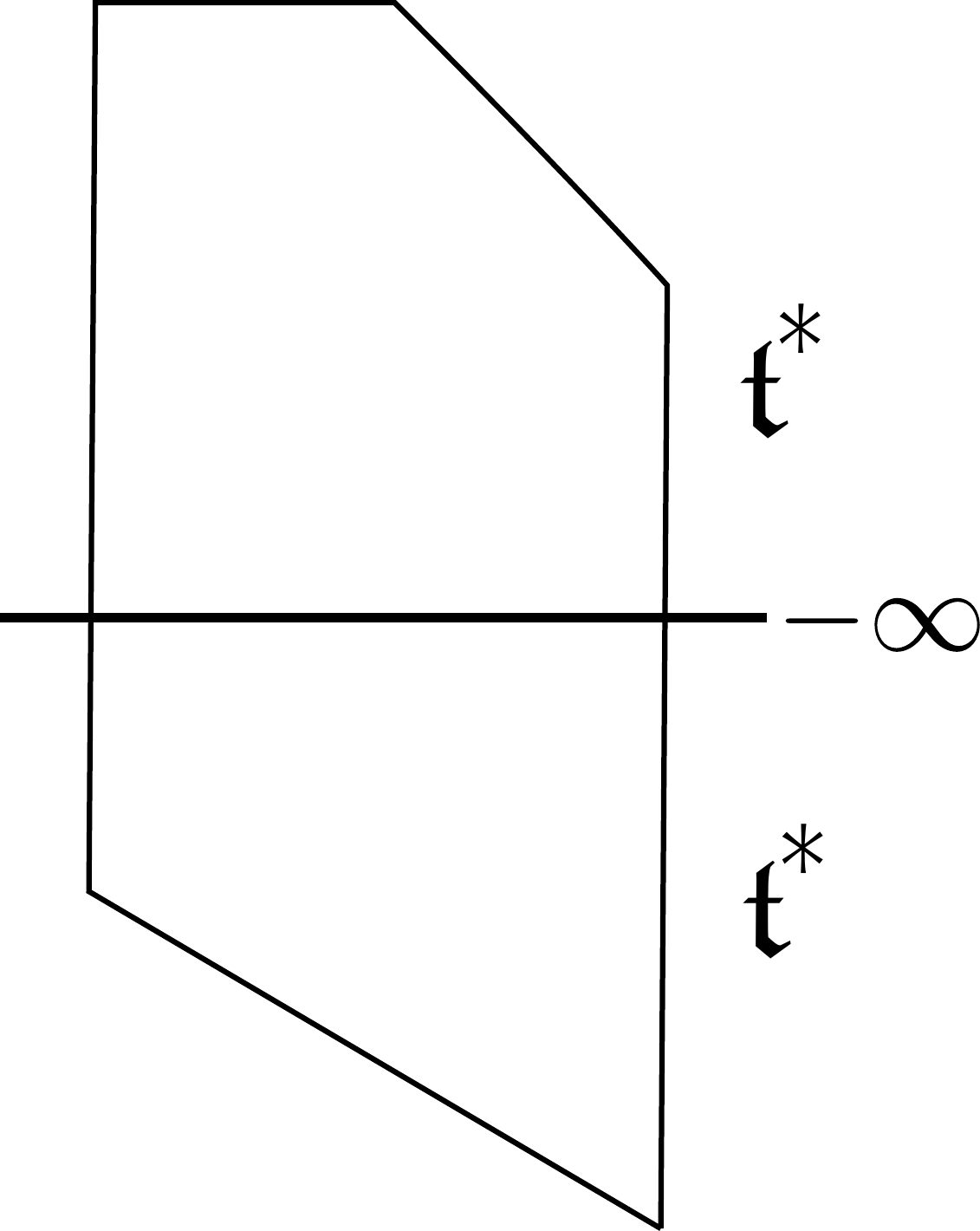}
\caption{Example of $b$-Delzant polytope.}
\label{fig:Delzant}
\end{figure}
Now we are able to state the $b$-Delzant type classification
theorem.
\begin{theorem}
The $b$-moment map
$$\{\text{$b$-toric manifolds }(M,Z,\omega)\}\longrightarrow\{\text{$b$-Delzant polytopes}\}$$
which sends each $b$-toric manifold to its $b$-moment map image is a
bijection. Here $b$-toric manifolds are considered up to equivariant
$b$-symplectomorphisms that preserves the moment map
\end{theorem}

For the proof of injectivity, see \cite{btoric}. For the proof of
surjectivity, one can simply consider a standard model
$X_{\Delta_Z}\times\mathbb{S}^2$ and then applying symplectic
cuttings, as it was done in \cite{btoric}.

\section{Construction from toric manifolds}\label{section construction}

Now we come to the main part of the paper. In this part we come up
with a way constructing $b$-toric manifold from the information of
the $b$-Delzant polytope. This is mainly done by doing surgeries to
toric manifolds and gluing them together.

\noindent\textbf{Caution.}
It should be pointed out that the image of the moment map, i.e., the
Delzant polytope is allowed to be transformed under
$\rm{SL}(n,\mathbb{Z})$ and be translated by some constants, because
the action is $\mathbb{T}^n$ which is abelian. The later one can
also be done for $b$-moment map image.

\subsection{Observation}
In order to get a feeling for this construction, we first consider a
simple example. For example, consider Figure \ref{fig:Delzant}. One
may wonder is there any link between the $b$-toric manifold
corresponding to Figure \ref{fig:Delzant} and the toric manifold
with Delzant polytope as below.
\begin{figure}[ht]
\includegraphics[width=2cm,height=3.5cm]{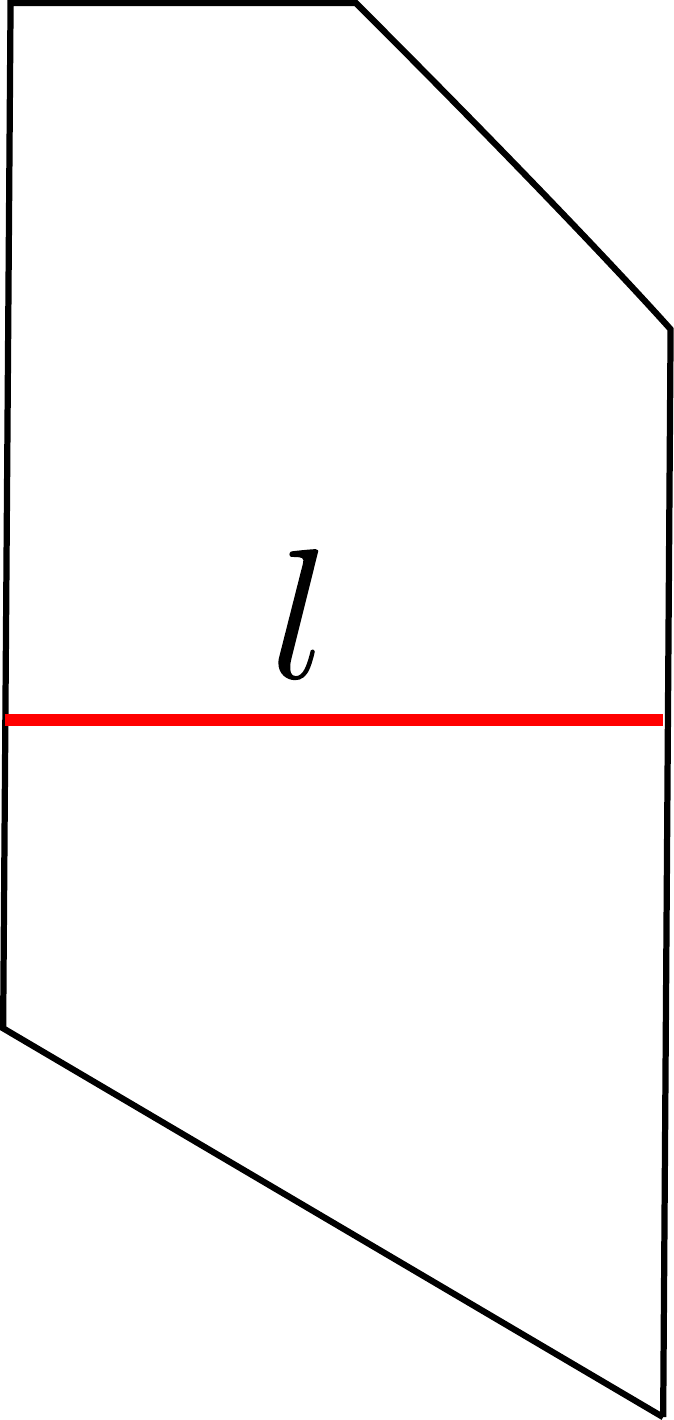}
\caption{A usual Delzant polytope $\Delta$ in
$\mathbb{R}^2=\rm{Lie}(\mathbb{T}^2)^*$.} \label{fig:normal}
\end{figure}

We denote the corresponding toric manifolds as $X_\Delta$, its
symplectic form as $\omega$ and moment map as $\mu$. If we put
ourself in the so-called symplectic coordinates, i.e. $(\Delta\times
\mathbb{T}^2,(x,\xi))$ with symplectic form $dx\wedge d\xi$ (see,
for example, \cite{abreu1}), it seems that we should blow up the
symplectic form $\omega$ around the hypersurface $l$. This can be
easily done, one just needs to split the polytope to three part as
below.
\begin{figure}[ht]
\includegraphics[width=4cm,height=3.5cm]{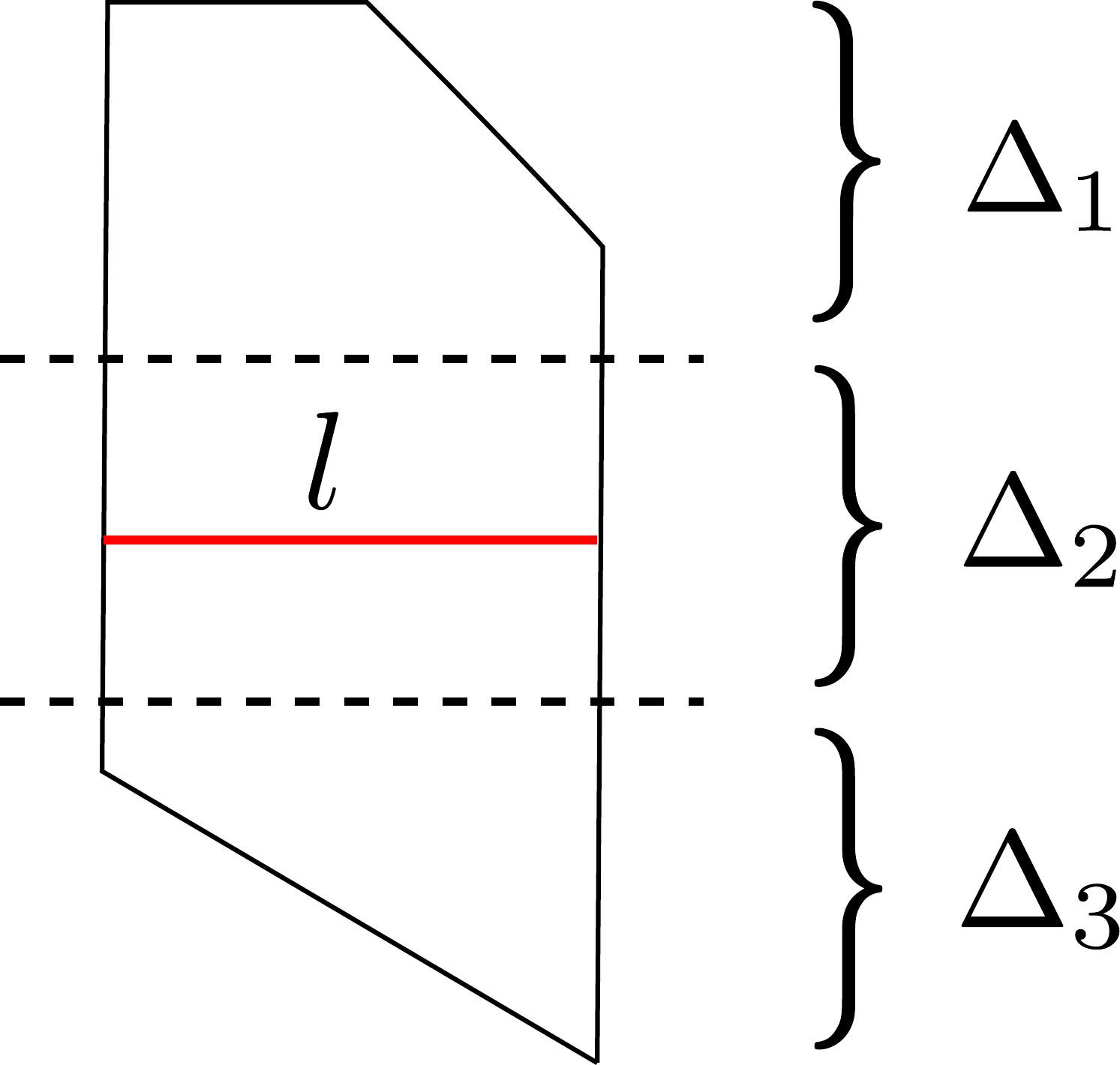}
\caption{A partition of $\Delta$}
\label{fig:normal2}
\end{figure}

Now we can simply keep the symplectic form in part $\Delta_1$ and
denote it by $\omega_1$, replace the symplectic form in $\Delta_2$
by a $b$-symplectic form $\omega_2$ with singular hypersurface $l$
and in part $\Delta_3$, if we use $x_1$ to denote the horizontal
coordinate, $x_2$ to denote the vertical coordinate, we should
replace $dx_2$ by $-dx_2$ in the symplectic form. We denote the
resulted symplectic form in $\Delta_3$ by $\omega_3$. Then simply
use partition of unity to add the three forms together, i.e.
$$\tilde{\omega}=\rho_1\omega_1+\rho_2\omega_2+\rho_3\omega_3.$$

But we will find many problems in this kind of operation:
\begin{itemize}
\item From the symplectic coordinate one still needs to do compactification to get an actual toric variety. Is this "hypersurface" $l$ in the symplectic coordinate actually
a hypersurface in $X_\Delta$ ?
\item For the same reason as above, we also need to worry about
wether the reversed symplectic form in $\Delta_3$ is compatible with
our compactification.
\end{itemize}

As we shall see, it will be a hypersurface and the first question
can be solved. However, the second one is an actual problem, and the
following two examples will help us to see why.

\begin{example}
This is a very simple one. Consider the $b$-toric manifold
$(\mathbb{S}^2\times\mathbb{S}^2,Z=\{h_2=0\},\omega=dh_1\wedge
d\theta_1+\frac{1}{h_2}dh_2\wedge d\theta_2)$, with toric
$\mathbb{T}^2$ action by rotation along the flows generated by
$\frac{\partial}{\partial\theta_1}$ and
$\frac{\partial}{\partial\theta_1}$. It is easy to see that the
$b$-moment map should be $h_1+\log|h_2|$, and the $b$-Delzant
polytope corresponding to it should be as the Figure
\ref{fig:doubulesphere}.

For this $b$-toric manifold, if we start with the toric manifold
$(\mathbb{S}^2\times\mathbb{S}^2,\omega=dh_1\wedge
d\theta_1+dh_2\wedge d\theta_2)$, doing the pinching process as
described above, we can get the desired $b$-toric manifold
$\mathbb{S}^2\times\mathbb{S}^2$. Indeed, if we split it to three
parts as before, and select the partition function carefully, we can
get $b$-toric $\mathbb{S}^2\times\mathbb{S}^2$.
\begin{figure}[ht]
\includegraphics[width=4.5cm,height=4cm]{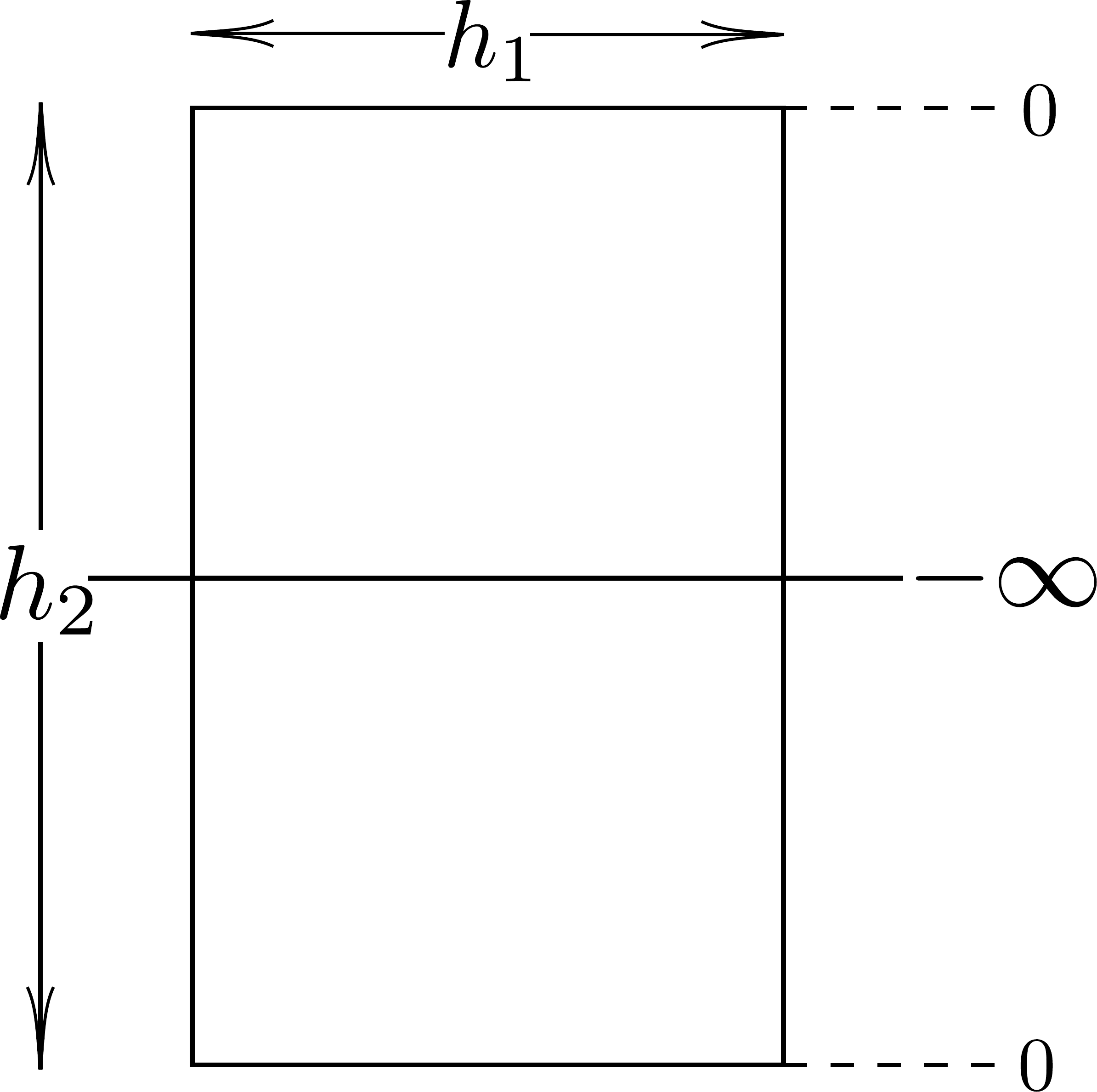}
\caption{$b$-Delzant polytope of $\mathbb{S}^2\times\mathbb{S}^2$.}
\label{fig:doubulesphere}
\end{figure}
\end{example}

\begin{example}
This is a negative example due to the second problem. Consider the
following $b$-Delzant polytopes corresponding to a $b$-toric
manifold $M$ and $M'$ respetively.
\begin{figure}[ht]
\includegraphics[width=8cm,height=4cm]{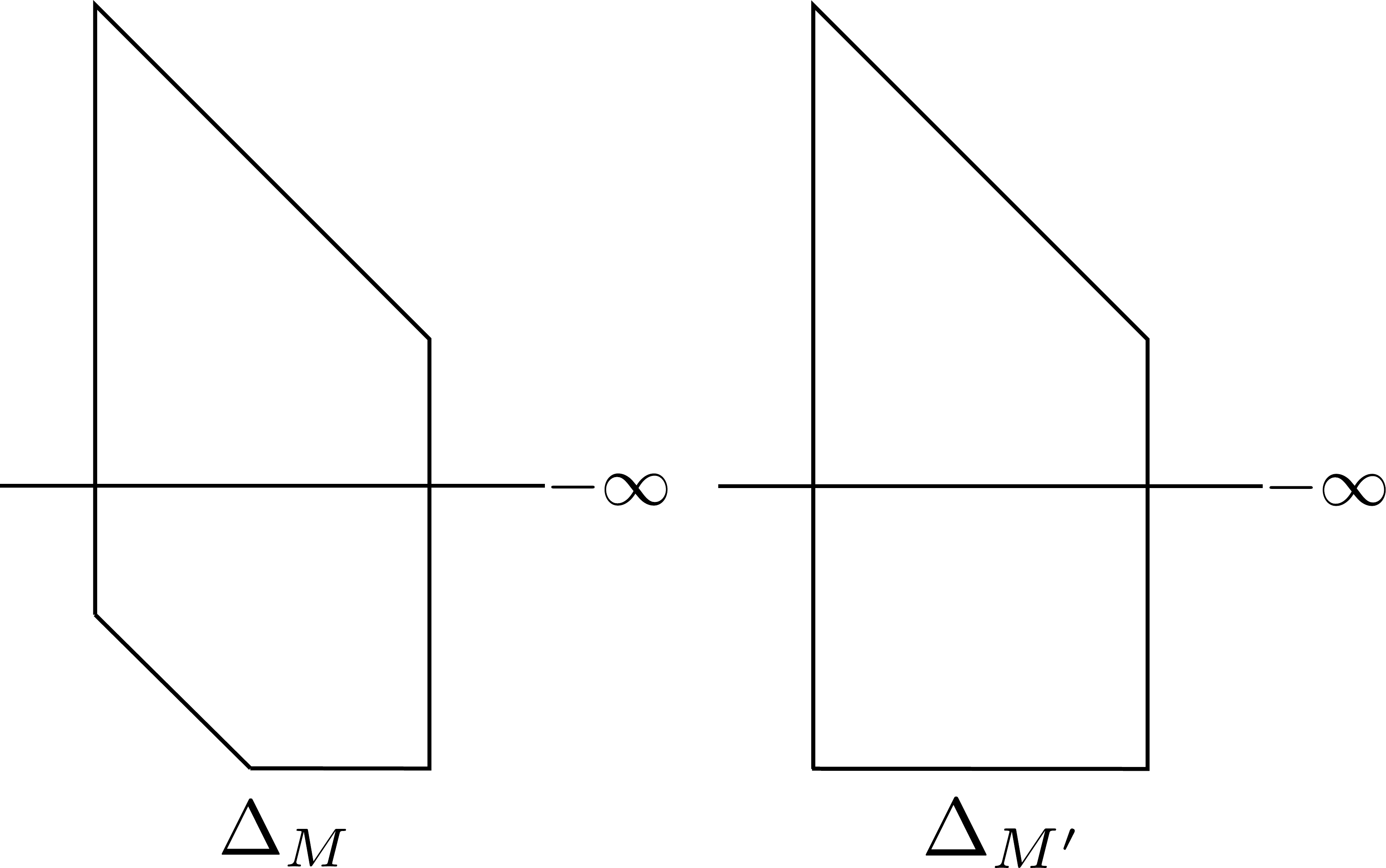}
\caption{$b$-Delzant polytopes of $M$ and $M'$.}
\label{fig:singularblowhizebruch}
\end{figure}

The corresponding usual Delzant polytope of $M$ is the moment
polytope of Hirzebruch surface
$H_1=\mathbb{P}({\mathcal{O}(1)\oplus1})\simeq\mathbb{CP}^2\#\overline{\mathbb{CP}^2}$
chopped at one vertex, where $\mathcal{O}(1)\oplus1$ means the
direct sum of tautological line bundle over $\mathbb{CP}^1$ and
trivial bundle over $\mathbb{CP}^1$. See for example \cite{silva2}.
According to \cite{guilleminsternberg3} this will be equivalent to
do symplectic blow up at the corresponding fixed point. That is to
say, the toric manifold corresponds to
$\mathbb{CP}^2\#\overline{\mathbb{CP}^2}\#\overline{\mathbb{CP}^2}.$
And hence, the base manifold of $M$ should be
$\mathbb{CP}^2\#\overline{\mathbb{CP}^2}\#\overline{\mathbb{CP}^2}$.\

However, for the same reason $M$ should be the blow up of $M'$ at
the corresponding vertex. But this time, due to the fact that the
symplectic form blow the $-\infty$ line is reversed, so it should be
$\mathbb{CP}^2\#\overline{\mathbb{CP}^2}\#{\mathbb{CP}^2}.$

Now if the construction above works, we would conclude that
$$\mathbb{CP}^2\#\overline{\mathbb{CP}^2}\#\overline{\mathbb{CP}^2}\simeq\mathbb{CP}^2\#\overline{\mathbb{CP}^2}\#{\mathbb{CP}^2}.$$
which is absurd.
\end{example}

From the example above, we can see that the essential reason for
this to be wrong is that the orientation of one half of it is
reversed.

\begin{remark}
There is something more one can talk about
$\mathbb{CP}^2\#\overline{\mathbb{CP}^2}\#{\mathbb{CP}^2}$. By
computing Seiberg-Witten invatiant in \cite{taubes} it was showed
that it does not admit any symplectic structure which is compatible
with its orientation. One can also refer to \cite{moore}. So it is
impossible for it coming from a toric manifold. This also provide us
with an example which admits $b$-symplectic structure but does not
admit a (orientation compatible) symplectic structure.
\end{remark}

\subsection{Toric submanifolds}

In this subsection we deal with the first problem appears in the
above subsection. We only care about a special kind of Delzant
polytope, which is useful in our construction.

\begin{definition}
Given $2n$ dimensional toric manifold $(M,\mathbb{T}^n)$ and a $2m$
dimensional symplectic submanifold $M'$, we say $M'$ is a
\textbf{toric submanifold} of $M$ if there exists a
$\mathbb{T}^m\subset\mathbb{T}^n$ such that the action by
$\mathbb{T}^m$ preserves $M'$.
\end{definition}

\begin{definition}
We say a Delzant polytope $\Delta$ is \textbf{parallel} if it
satisfies the condition that there exists one hyperplane $F$
intersecting $\Delta$, and each $1$ dimensional facet $l$ in
$\Delta$ which intersects $F$, intersects $F$ orthogonally. Such a
hyperplane will be called \textbf{parallel hyperplane}. We denote a
parallel Delzant polytope by $(\Delta,F)$.
\end{definition}

Here is one example of parallel Delzant polytope.
\begin{figure}[ht]
\includegraphics[width=4cm,height=4cm]{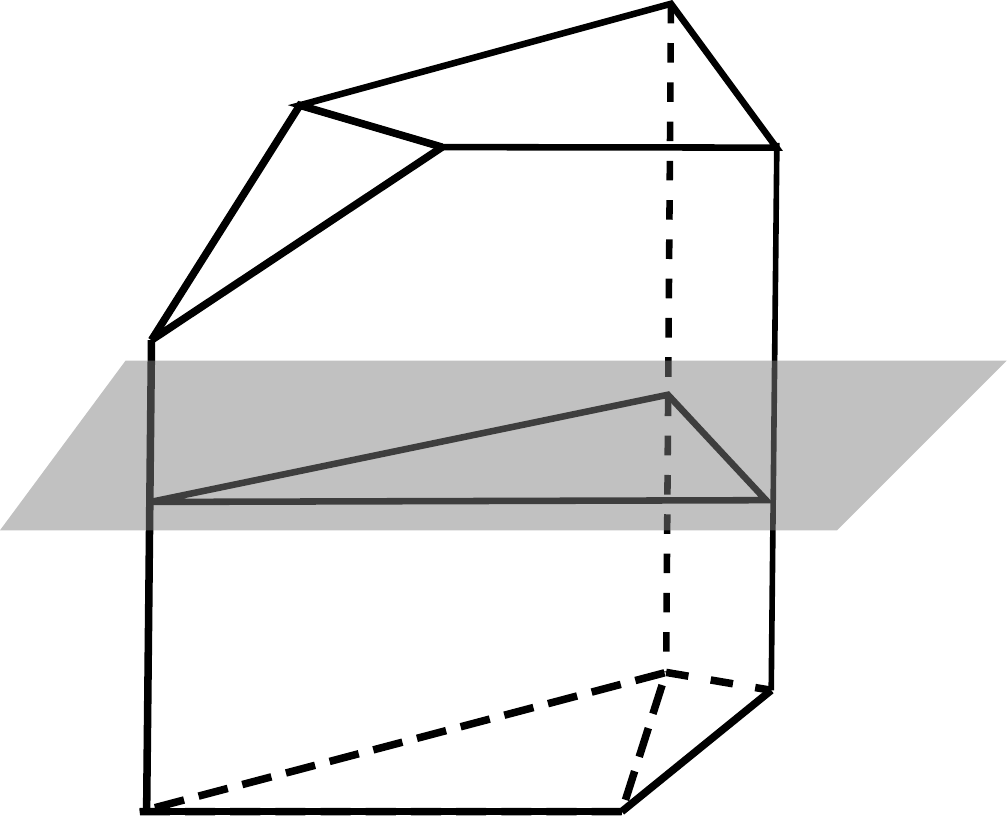}
\caption{Example of parallel Delzant polytope.}
\label{fig:parallel}
\end{figure}

\begin{lemma}
Let $\Delta_F=F\cap\Delta$. Then $\Delta_F$ is also Delzant.
\end{lemma}
\begin{proof}
Let $\overline{v}$ be a vertex in $\Delta_F$. Assume $\Delta$ is of
dimension $n$. What we need to prove is that, there exists a basis
$\overline{k}_1,\ldots,\overline{k}_{n-1}$ of $\mathbb{Z}^{n-1}$
such that each edge starting from $\overline{v}$, can be written as
$\overline{v}+t\overline{k}_i$. In fact, for any vertex
$\overline{v}$ by definition there must exist an edge $l$ of
$\Delta$ which intersects $\Delta$ at $\overline{v}$. We consider a
vertex of $\Delta$ which is located at $l$. Then there exists basis
of $\mathbb{Z}^{n}$ $k_1,\ldots,k_{n}$. One of them must be along
$l$. Now project to $F$, we get the desired $\overline{k}_i$.
\end{proof}

\begin{theorem}
Given a parallel Delzant polytope $(\Delta,F)$ and the associated
toric manifold $X_{\Delta}$. Assume its moment map is
$\mu_{\Delta}$. Then for any parallel hyperplane $F$ which
intersects $\Delta^{\circ}$, $\mu^{-1}(F\cap\Delta)$ will be a
hypersurface of $X_{\Delta}$. Moreover, this is isomorphic to
$X_{\Delta_F}\times\mathbb{S}^1$.
\end{theorem}
\begin{proof}
We can prove this through the original construction from Delzant
polytope to get toric variety. Recall that the Delzant polytope is
defined by $\langle x, u_i\rangle\geq \lambda_i$, for
$i=1,\cdots,d$, where $u_i$ is primitive element in the lattice
$\mathbb{Z}^n$, due to the fact that $\Delta$ is Delzant. Then we
have an exact sequence,
\begin{align*}
0\longrightarrow\mathfrak{n}\stackrel{\iota}{\longrightarrow}\mathbb{R}^d\stackrel{\beta}{\longrightarrow}\mathbb{R}^n\longrightarrow0.
\end{align*}
Where $\beta$ is the map $e_i\mapsto u_i$, $\mathfrak{n}$ is the
kernel of $\beta$ and $\iota$ is the inclusion. Dualize it will
yield a sequence,
$$0\longrightarrow(\mathbb{R}^n)^*\stackrel{\beta^*}{\longrightarrow}(\mathbb{R}^d)^*\stackrel{\iota^*}{\longrightarrow} \mathfrak{n}^*\longrightarrow0.$$
Then one can consider the usual action of $\mathbb{T}^d$ on
$\mathbb{C}^d$ by rotation with moment map
$J(z)=\frac{1}{2}(|z_1|^2,\ldots,|z_d|^2)+(\lambda_1,\ldots,\lambda_d)$,
and doing symplectic reduction with respect to $0$ level and the $N$
action, where $N$ is the Lie group generated by $\mathfrak{n}$. That
is to say, we quotient $\{z\in\mathbb{C}^d|\iota^*\circ J(z)=0\}$ by
$N$. And the corresponding toric manifold is
$X_\Delta=\{z\in\mathbb{C}^d|\iota^*\circ J(z)=0\}/N$. See, for
example, \cite{guillemin}.

However, $\iota^*\circ J(z)=0$ if and only if $J(z)=\beta^*(J_N(z))$
for some unique $J_N(z)\in(\mathbb{R}^n)^*$ by the exact sequence.
In fact, $[z]\in\{z\in\mathbb{C}^d|\iota^*\circ J(z)=0\}/N$,
$[z]\mapsto J_N(z)$ is exactly the moment map of
$X_\Delta\simeq\{z\in\mathbb{C}^d|\iota^*\circ J(z)=0\}/N$ (Note
that this map is well-defined). Taking inner product, we will have
\begin{align*}
\langle J(z),e_i\rangle&=\langle\beta^*(J_N(z)),e_i\rangle\\
&=\langle J_N(z),u_i\rangle\\
&=\frac{1}{2}|z_i|^2+\lambda_i.
\end{align*}
That is to say, the distance of the moment map image $J_N([z])$ of
$[z]\in X_\Delta$ to each codimension 1 facet of $\Delta$ will be
$$\langle J_N(z),u_i\rangle-\lambda_i=\frac12|z_i|^2.$$

In the parallel Delzant case, we use the following notation. We use
$v_i$, $i=1,\ldots,s$ to denote the vectors in $\{u_i\}$ that are
parallel to the hyperplane $F$, and we use $r_j$, $j=1,\ldots,t$ to
denote those who are not parallel to $F$. Arrange $\{v_i\}$ and
$\{r_j\}$ such that $v_1,\ldots,v_s,r_1,\ldots,r_t$ is the same as
$u_1,\dots,u_d$. Now, by our hypothesis, $F\cap\Delta$ will have
positive distance to each codimension 1 facet of $\Delta$
corresponding to $r_j$. Hence for those $z\in J_N^{-1}(F)$,
$|z_{s+j}|\neq0$.

Let $t^\perp\in\mathbb{R}^n=\rm{Lie}\,\mathbb{T}^n$ be a unit vector
such that ker$\,t^\perp$ in $\mathfrak{t}^*$ is parallel to $F$.

\begin{claim}
$\mathbb{S}^1=\{\exp kt^\perp|k\in\mathbb{R}\}$ acts freely on
$J_N^{-1}(F)$
\end{claim}
This is simple. Because as a set,
$J_N^{-1}(F\cap\Delta^{\circ})=F\cap\Delta^{\circ}\times\mathbb{T}^n$,
and the action on it must be free. At the point $x\in J_N^{-1}(F\cap
P)$ for some facet $P$ of $\Delta$, the isotropy group will be
$\mathbb{T}^P\subset\mathbb{T}^n$, where $\mathbb{T}^P=\exp V$ and
$V\subset\rm{Lie}\,(\mathbb{T}^n)$ is the kernel of the linear space
corresponding to the affine space $P$. It is obvious that
$t^\perp\not\in V$.

Without loss of generality, we can do a $\rm{SL}(n,\mathbb{Z})$
transformation to make $t^\perp$ become $e_d$. This is because
$t^\perp$ is a vector along some edges of $\Delta$.

Now for any vector $e_{s+j}\in\mathbb{R}^d$ with $s+j<d$, by the
hypothesis, we can find $a_{s+j,i}$, $i=0,\ldots,s$ such that
$e_{s+j}-a_{s+j,0}t^\perp-a_{s+j,1}e_1-\ldots-a_{s+j,s}e_s\in
\mathfrak{n}$.
\begin{claim}
Because of the assumption we have $v_1,\ldots,v_s$ are all parallel
to $F$, the map $\beta$ can be factored to get

\begin{equation*}
\begin{tikzcd}
{\mathbb{R}^d\subset\text{span}\{v_1,\ldots,v_n\}} \arrow[rr, "\overline{\beta}"] \arrow[rrdd, "\beta"] &  & \mathbb{R}^{n-1} \arrow[dd, "\psi"] \\
                                                                                                        &  &                                     \\
                                                                                                        &  & \mathbb{R}^n                       
\end{tikzcd}
\end{equation*}

Then $\rm{ker}\,\overline{\beta}\subset\mathfrak{n}$ and
$\rm{ker}\,\overline{\beta}\oplus\rm{span}\,\{
e_{s+j}-a_{s+j,0}t^\perp-a_{s+j,1}e_1-\ldots-a_{s+j,s}e_s\}$ equals
to $\mathfrak{n}$.
\end{claim}
This claim is trivial, becasue given any element in $\mathfrak{n}$
we can use $\rm{span}\,\{
e_{s+j}-a_{s+j,0}t^\perp-a_{s+j,1}e_1-\ldots-a_{s+j,s}e_s\}$ to
annihilate its $e_{s+j}$ part, and the rest just consists of
elements in $\rm{ker}\,\overline{\beta}$.

Hence, by acting $\exp
t(e_{s+j}-a_{s+j,0}t^\perp-a_{s+j,1}e_1-\ldots-a_{s+j,s}e_s)\in N$
we can make, $z_{s+j}=|z_{s+j}|$, for $z\in J_N^{-1}(F)$ and
$s+j<d$.That is to say, we can forget about $e_{s+j}$ for $s+j<d$.
The last component $z_d$ is free because the $t^\perp$ action is
free and $z_d\neq0$. Thus to consider $J_N^{-1}(F)/N$, assuming
$\pi_1:\mathbb{C}^d\rightarrow\mathbb{C}^s$ is the projection to the
first $s$ components, it suffices to consider $\pi_1(J_N^{-1}(F))$
quotient by $\exp\rm{span}\{v_1,\ldots,v_s\}$. And
$J_N^{-1}(F)/N=(\pi_1(J_N^{-1}(F))/(\exp\rm{span}\{v_1,\ldots,v_s\}))\times\mathbb{S}^1$.

\begin{claim}
$\pi_1(J_N^{-1}(F))/(\exp\rm{span}\{v_1,\ldots,v_s\})$ is exactly
$X_{\Delta_F}$.
\end{claim}
After a similar argument in Claim 26, we can conclude that
$\exp\rm{span}\{v_1,\ldots,v_s\}$ preserves $\pi_1(J_N^{-1}(F))$,
and $(z_1,\ldots,z_s)\in\pi_1(J_N^{-1}(F))$ satisfies the condition
$\iota^*\circ(\frac{1}{2}(|z_1|^2,\ldots,|z_s|^2)+(\lambda_1,\ldots,\lambda_s))=0$,
from the construction we introduced above we see this is exactly
$X_{\Delta_F}$.
\end{proof}
\begin{remark}
There is a similar version of this theorem in the $b$-symplectic
case, see \cite{btoric}.
\end{remark}
\begin{remark}
Moreover, we can see easily that in this local model $X_{\Delta_F}\times\mathbb{S}^1$, each codimension $2$ submanifolds $X_{\Delta_F}\times\{\theta\}$ is a toric submanifold.
\end{remark}

\subsection{Local model for toric manifolds}

In this section we study more local property of a special kind of
toric variety $(X_{\Delta},\Delta)$. $\Delta$ will be more special
and such toric varieties will be our fundamental building blocks for
our construction.

\begin{definition}
We say a parallel Delzant polytope $(\Delta,F)$ is \textbf{strict
parallel}, if there exists a parallel hyperplane $F'$ for
$(\Delta,F)$ such that $F'\cap\Delta=F'\cap\partial\Delta$. $F'$ is
called the \textbf{strict parallel hyperplane}.
\end{definition}

For example, Figure \ref{fig:parallel} is not strict parallel with
respect to the hyperplane $F$. However, the following picture gives
us an example of strict parallel Delzant polytope.

\begin{figure}[ht]
\includegraphics[width=4cm,height=4cm]{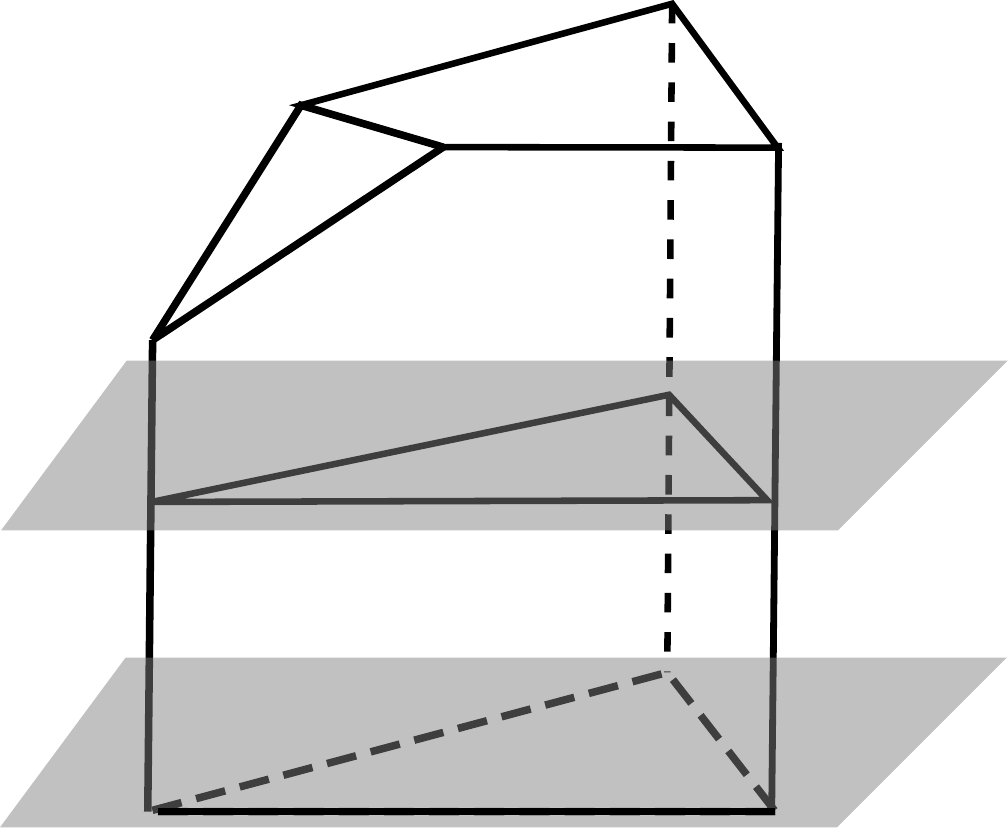}
\caption{Example of strict parallel Delzant polytope.}
\label{fig:strictparallel}
\end{figure}

We will study the local behavior of a strict parallel toric variety
$(\Delta,F)$ around the facet $F\cap\Delta$, where $F$ is a strict
parallel hyperplane. In the next subsection we will cut
$F\cap\Delta$ and then glue two toric manifold along it.

From now on, by a $\rm{SL}(n,\mathbb{Z})$ transformation, we can
always assume the strict parallel hyperplane is parallel to the
$t_1^*,\ldots,t_{n-1}^*$ plane in
$(\mathbb{R}^n)^*=(\rm{Lie}\,(\mathbb{T}^n))^*$. Moreover, we can
assume the entire Delzant polytope is above $F$, because the case
where the polytope is under $F$ is exactly the same. We also assume
the dimension of $\Delta$ is $n$ and the dimension of $\Delta_F$ is
$n-1$.

\begin{theorem}
Given a strict parallel toric manifold $(X_{\Delta},F)$, F a strict
parallel hyperplane, assume it has moment map $\mu_{\Delta}$. For a
suitable small neighborhood $U$ of $\Delta_F=\Delta\cap F$ in
$\Delta$, $\mu_\Delta^{-1}(U)$ will be symplectomorphic to
$X_{\Delta_F}\times\mathbb{S}^2{[-r,-r+\epsilon)}$. Here
$\mathbb{S}^2$ is a sphere with suitable radius $r$ and the common
symplectic form $dh\wedge d\theta$.
\end{theorem}
The following picture illustrates the essence of this theorem.
\begin{figure}[ht]
\includegraphics[width=7.5cm,height=4cm]{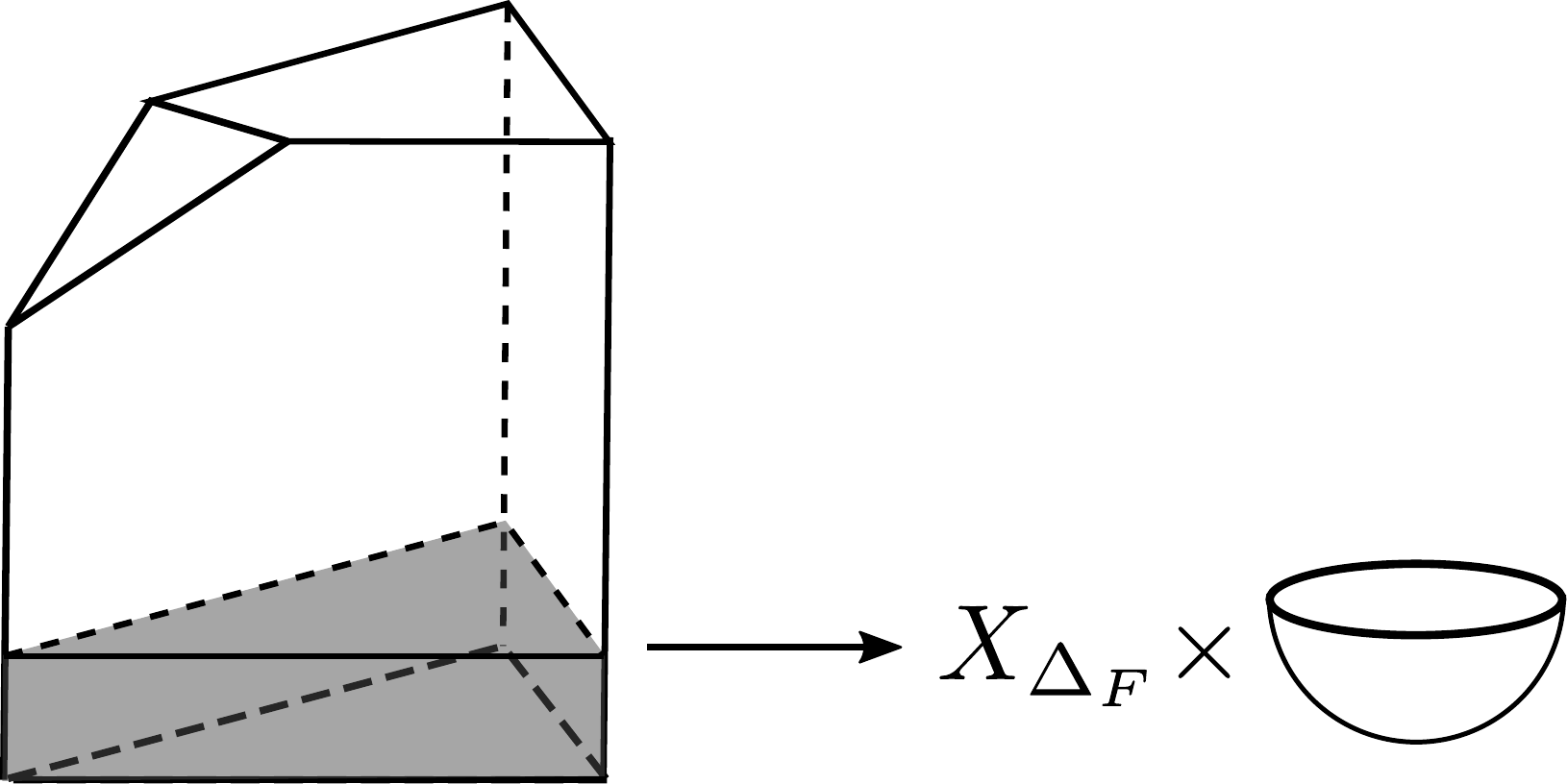}
\caption{Local model for neighborhood of strict hyperplane.}
\label{fig:localmodel}
\end{figure}

In fact, we can prove that there is a diffeomorphism very easily.
The crucial point is to prove this is a symplectomorphism. We will
apply symplectic cutting method. See \cite{lerman}.
\begin{proof}
Given a strict parallel Delzant polytope, there is a very simple way
to construct the desired toric manifold $X_\Delta$. We consider
$X_{\Delta_F}\times\mathbb{S}^2$, where $\mathbb{S}^2$ is a sphere
with suitable large radius, and is equipped with the common
symplectic form $dh\wedge d\theta$. Now the Delzant polytope
corresponding Delzant polytope of $X_{\Delta_F}\times\mathbb{S}^2$
will be $\Delta_F\times[-r,r]$. Because of the assumption that
$\Delta$ is strict parallel, up to adding or subtracting some
element in $(\mathbb{R})^*$, we can assume the strict parallel
hyperplane $F=\{t^*_n=-r\}$, and then one can get the polytope
$\Delta$ by cutting $\Delta_F\times[-r,r]$ at the top. As shown in
the following picture.

\begin{figure}[ht]
\includegraphics[width=8.3cm,height=3.8cm]{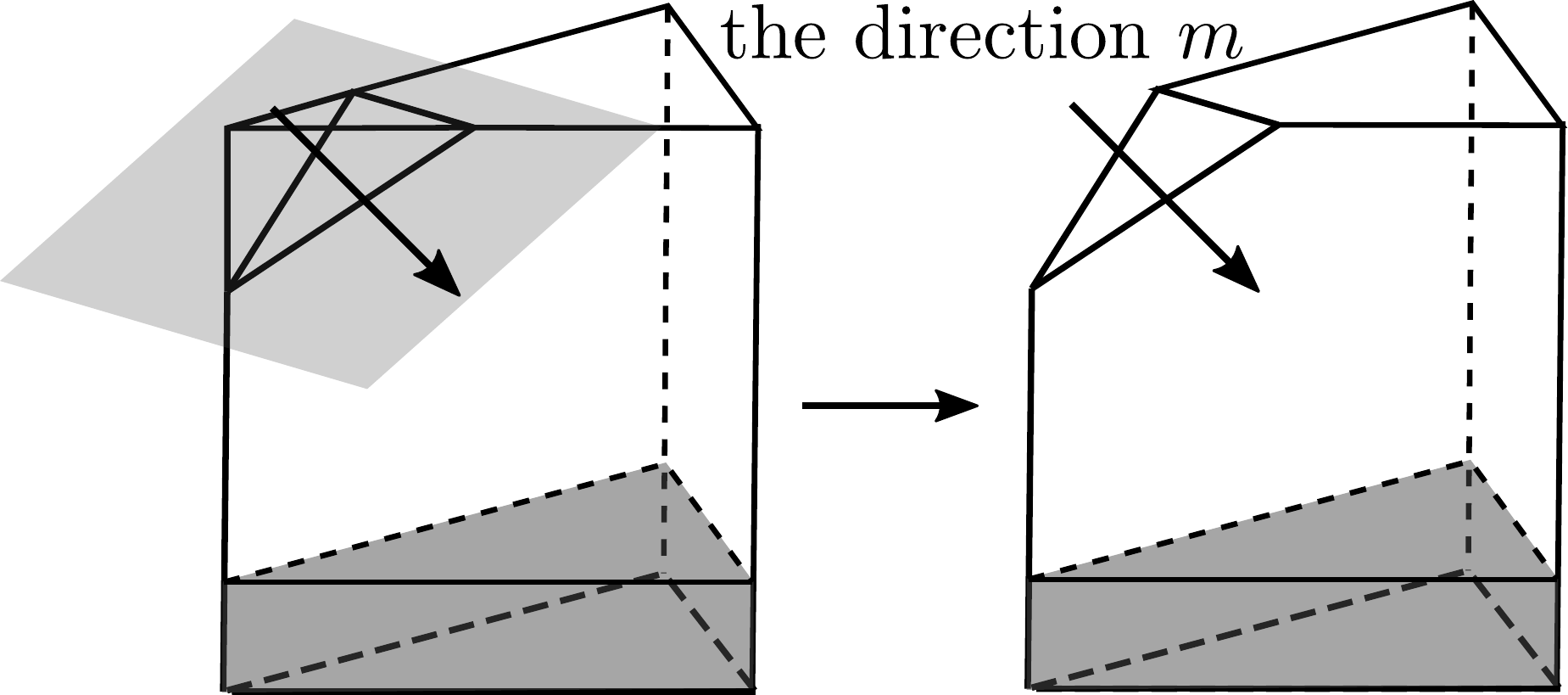}
\caption{Symplectic cutting.}
\label{fig:cut}
\end{figure}

We select positive $\epsilon$ small enough such that
$\Delta_F\times[-r,-r+\epsilon)$ is not cut under our cutting
process.

For the sake of simplicity, we can assume that we only cut
$\Delta_F\times[-r,r]$ once. The general case can be got by
induction.

Now, assume the cut is done with the
$m\in\mathbb{R}^n=\rm{Lie}\,(\mathbb{T}^n)$ direction and at level
$\delta$ (such a direction will be rational because the polytope we
get is Delzant). That is to say, we consider the
$\mathbb{S}^1=\{\exp (tm)|t\in\mathbb{R}\}$ action on
$X_{\Delta_F}\times\mathbb{S}^2\times\mathbb{C}$ by
$$\exp(tm):(x,z)\mapsto(\exp(tm)x,\exp(-it)z)\text{, for }x\in X_{\Delta_F}\times\mathbb{S}^2\text{ and }z\in\mathbb{C}.$$
The moment map for this action will be
$$\langle\mu,m\rangle-\frac{1}{2}|z|^2-\delta.$$
Here we write $\mu$ for the moment map of
$X_{\Delta_F}\times\mathbb{S}^2$. Take symplectic reduction with
respect to this action we will get the desired symplectic cut. And
in this case what we get is $X_{\Delta}$. In fact, the corresponding
symplectic reduction is just
$\{\langle\mu,m\rangle=\frac{1}{2}|z|^2+\delta\}/\mathbb{S}^1$.
Hence we are actually keep the part $\langle\mu,m\rangle\geq\delta$
in $X_{\Delta_F}\times\mathbb{S}^2$ and then compatify the boundary
to get the desired toric manifold.

Now we will consider this symplectic reduction process more
specifically. We denote the moment map for the $\mathbb{S}^1$ action
on $X_{\Delta_F}\times\mathbb{S}^2\times\mathbb{C}$ by
$J(x,z)=\langle\mu(x),m\rangle-\frac12|z|^2-\delta$, the quotient
map by $\pi:J^{-1}(0)\rightarrow X_\Delta$ and the inclusion map by
$\iota:J^{-1}(0)\rightarrow
X_{\Delta_F}\times\mathbb{S}^2\times\mathbb{C}$. We use $\omega$ for
the symplectic form on $X_{\Delta_F}\times\mathbb{S}^2$ and $\mu$
for its moment map. Then the symplectic form $\omega_{X_\Delta}$ on
$X_\Delta$ will satisfy
$$\pi^*\omega_{X_\Delta}=\iota^*(\omega+dx\wedge dy),$$
and the moment maps are related by
$$\mu_{\Delta}\circ\pi=\overline{\mu}\circ\iota.$$
Here by an abuse of notation, we use $\overline{\mu}$ to denote the
moment map of the $\mathbb{T}^n$ action on
$X_{\Delta_F}\times\mathbb{S}^2\times\mathbb{C}$ by acting on the
first two components.

\begin{claim}
The symplectic form in the neighborhood
$\mu^{-1}_\Delta(\Delta_F\times[-r,-r+\epsilon))$ of
$\mu^{-1}_\Delta(\Delta_F)$ is unchanged. More specifically, we have
a symplectomorphism
\begin{align*}
X_{\Delta_F}\times\mathbb{S}^2[-r,-r+\epsilon)&=\mu^{-1}(\Delta_F\times[-r,-r+\epsilon))\\
&\simeq\mu^{-1}_\Delta(\Delta_F\times[-r,-r+\epsilon)).
\end{align*}
\end{claim}

Indeed, the first equality in the above is trivial. For the
symplectomorphism, let the point $(x,z)\in J^{-1}(0)\subset
X_{\Delta_F}\times\mathbb{S}^2\times\mathbb{C}$, with
$x\in\mu^{-1}(\Delta_F\times[-r,-r+\epsilon))\subset
X_{\Delta_F}\times\mathbb{S}^2$. Because of the fact that $\mu(x)$
is away from the cutting hyperplane, we have
$\frac{1}{2}|z|^2=\langle\mu(x),m\rangle-\delta>0$. Hence we will
have

\begin{align*}
\mu^{-1}_\Delta(\Delta_F\times[-r,-r+\epsilon))&=\pi(\overline{\mu}^{-1}(\Delta_F\times[-r,-r+\epsilon))\cap
J^{-1}(0))\\
&=(\overline{\mu}^{-1}(\Delta_F\times[-r,-r+\epsilon))\cap
J^{-1}(0))/\mathbb{S}^1\\
&=\mu^{-1}(\Delta_F\times[-r,-r+\epsilon)).
\end{align*}
So we have a diffeomorphism.

Now we only need to examine the sympelctic form. Given two tangent
vectors $X$ and $Y$ of
$\mu^{-1}_\Delta(\Delta_F\times[-r,-r+\epsilon))$ at a point $x$, by
the symplectic cutting construction we will have $(x,z)\in
J^{-1}(0)\subset X_{\Delta_F}\times\mathbb{S}^2\times\mathbb{C}$. We
can select $z$ to be real and positive, because of the
$\mathbb{S}^1$ action. Now if $\gamma_X(t)$ is an integral curve in
$\mu^{-1}_\Delta(\Delta_F\times[-r,-r+\epsilon))$, then we have a
lift of the integral curve in $J^{-1}(0)$ $(\gamma_X(t),z_X(t))$,
with the first component $\gamma_X(t)$ and the second component
always real. Similarly we can do this for $Y$. Hence we have
\begin{align*}
\omega_\Delta(X,Y)&=\omega_\Delta\left(\left.\frac{d}{dt}\right|
_{t=0}(\gamma_X(t),z_X(t)),\left.\frac{d}{dt}\right|_{t=0}(\gamma_Y(t),z_Y(t))\right)\\
&=(\omega+dx\wedge
dy)\left(\left.\frac{d}{dt}\right|_{t=0}(\gamma_X(t),z_X(t)),\left.\frac{d}{dt}\right|_{t=0}(\gamma_Y(t),z_Y(t))\right)\\
&=\omega(X,Y)+dx\wedge
dy\left(\left.\frac{d}{dt}\right|_{t=0}z_X(t),\left.\frac{d}{dt}\right|_{t=0}z_Y(t)\right)\\
&=\omega(X,Y).
\end{align*}
The second term vanishes because $z_X(t)$ and $z_Y(t)$ are always
real. And the vector fields they generate are always along with
$\frac{\partial}{\partial x}$.

Till now we have already finished the proof. One just needs to
select the neighborhood $U$ as $\Delta_F\times[-r,-r+\epsilon)$. And
$\mu_\Delta^{-1}(\Delta_F\times[-r,-r+\epsilon))$ is
symplectomorphic to
$X_{\Delta_F}\times\mathbb{S}^2[-r,-r+\epsilon)$.
\end{proof}

The below corollary directly follows from the proof above.

\begin{corollary}
In the local model proved above, the $\mathbb{T}^n$ action splits to
$\mathbb{T}^{n-1}\times\mathbb{S}^1$ action, where
$\mathbb{T}^{n-1}$ action is the toric action on $X_{\Delta_F}$ and
$\mathbb{S}^1$ action is the usual action on a sphere, i.e.,
rotation along the $\theta$ parameter.
\end{corollary}

\subsection{Cutting and gluing}

Have already determined the local structure of our building blocks,
we can now start the cutting and gluing process.

\subsubsection{cutting}

Our building blocks consist of strict parallel Delzant polytopes.
Now given a strict parallel toric manifold $X_{\Delta}$, with
$\Delta$ strict parallel and $F$ the strict parallel hyperplane.
Still, without loss of generality we assume $F$ is parallel to the
$t^*_1,\ldots,t^*_{n-1}$ plane in $(\mathbb{R}^n)^*$.

If the Delzant polytope $\Delta$ is over $F$, we cut $X_\Delta$ at
the bottom by $X_{\Delta_F}$. That is to say, we consider the local
model proved in the above subsection,
$X_{\Delta_F}\times\mathbb{S}^2[-r,-r+\epsilon)$. And we cut
$X_{\Delta}$ by excluding
$$X_{\Delta_F}=X_{\Delta_F}\times\{-r\}\subset
X_{\Delta_F}\times\mathbb{S}^2[-r,-r+\epsilon)\subset X_{\Delta}.$$
This open symplectic manifold will be denoted by
$\leftidx{_{\rm{cut}}}X_{\Delta}$. Similarly one can cut $X_\Delta$
at the top, if there is a strict hyperplane $F$ such that $\Delta$
is under $F$. We denote this by $\leftidx{^{\rm{cut}}}X_{\Delta}$.
The following picture illustrates the cutting process for the first
case.

\begin{figure}[ht]
\includegraphics[width=12cm,height=3.5cm]{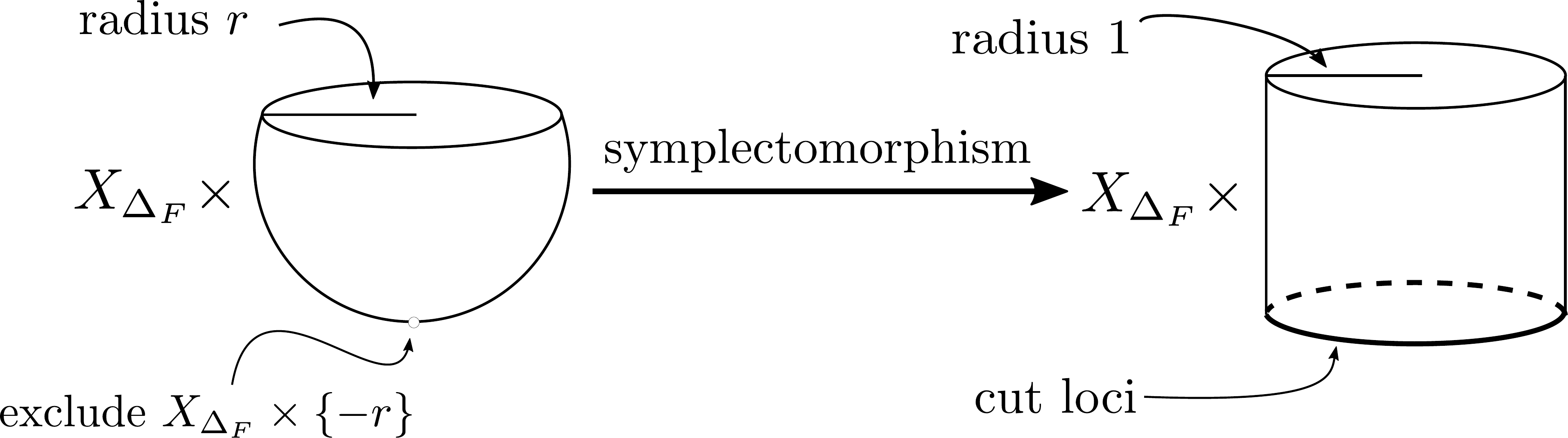}
\caption{Lower cutting.} \label{fig:cutting}
\end{figure}
Here the sphere and the cylinder are all equipped with symplectic
form $dh\wedge d\theta$.

\subsubsection{gluing}

In this part we glue two open symplectic manifold
$\leftidx{_{\rm{cut}}}{X}_{\Delta_1}$ and
$\leftidx{^{\rm{cut}}}{X}_{\Delta_2}$ together, where $\Delta_i$ are
strict parallel Delzant polytopes, and they have the same
$\Delta_F$.

By the local model above, we are just gluing two cylinders. There
are two ways to do it. The first one is gluing the cylinders with
the same orientation (recall as symplectic manifolds they all carry
an orientation). The other one is gluing with reversed orientation.
\begin{definition}
The orientation preserved gluing is defined to be
$$\leftidx{_{\rm{cut}}}{X}_{\Delta_1}\sqcup X_{\Delta_F}\times\mathbb{S}^1\sqcup\leftidx{^{\rm{cut}}}{X}_{\Delta_2},$$
with the obvious smooth structure, making the cylinders smoothly
pinched together with the same orientation. The orientation reserved
gluing is defined to be
$$\leftidx{_{\rm{cut}}}{X}_{\Delta_1}\sqcup X_{\Delta_F}\times\mathbb{S}^1\sqcup\leftidx{^{\rm{cut}}}{X}_{\Delta_2},$$
with the obvious smooth structure, making the cylinders smoothly
pinched together with the opposite orientation.
\end{definition}
The following picture show the two different kinds of gluing.

\begin{figure}[ht]
\includegraphics[width=12cm,height=6cm]{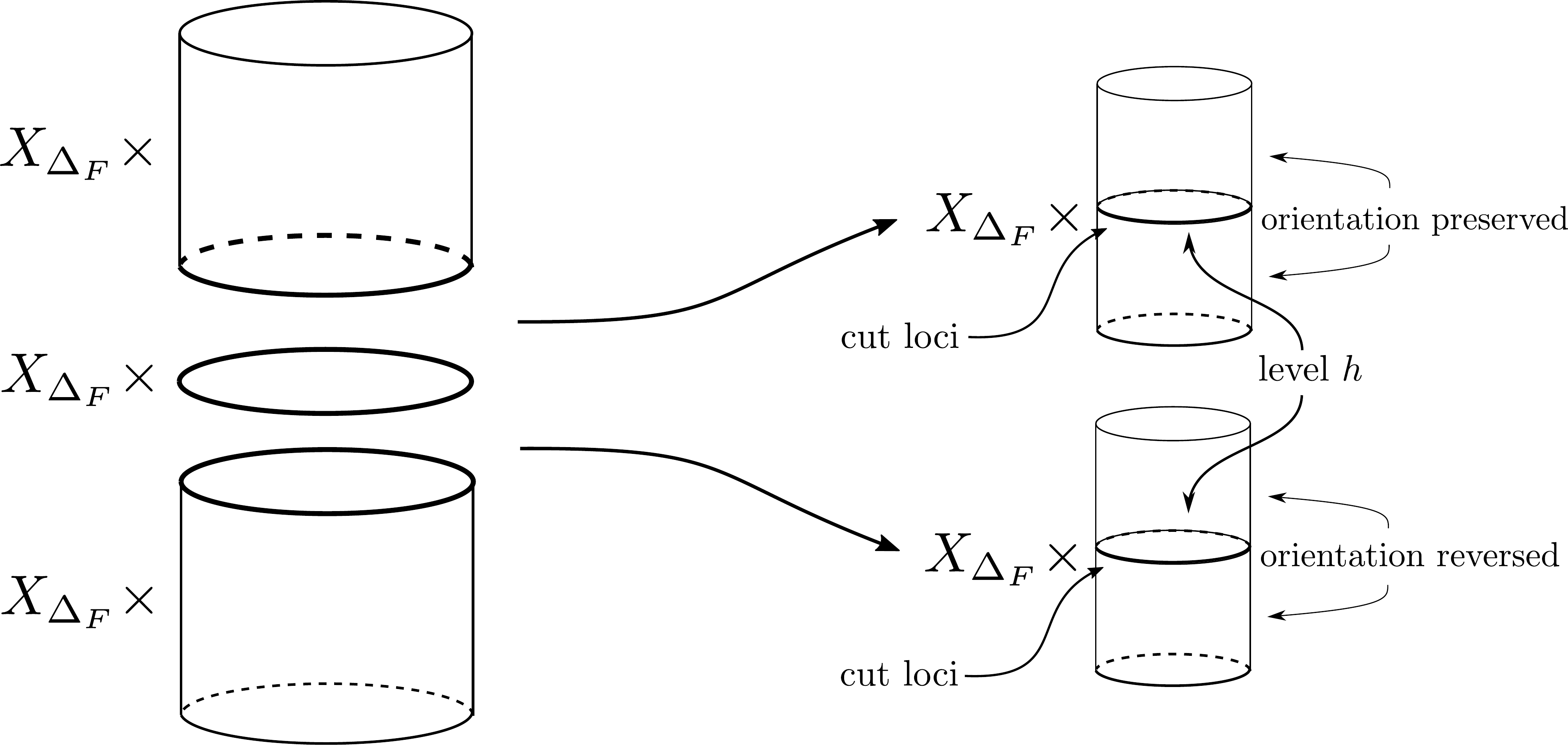}
\caption{Different way of gluing.} \label{fig:gluing}
\end{figure}


The same orientation glued manifold will be denoted as
$\leftidx{_{\rm{cut}}}{X}_{\Delta_1}\cup\leftidx{^{\rm{cut}}}{X}_{\Delta_2}$
and the orientation reversed gluing will be denoted by
$\leftidx{_{\rm{cut}}}{X}_{\Delta_1}\cup\overline{\leftidx{^{\rm{cut}}}{X}_{\Delta_2}}$.

\begin{theorem}
$\leftidx{_{\rm{cut}}}{X}_{\Delta_1}\cup\leftidx{^{\rm{cut}}}{X}_{\Delta_2}$
will carry a toric manifold structure. And
$\leftidx{_{\rm{cut}}}{X}_{\Delta_1}\cup\overline{\leftidx{^{\rm{cut}}}{X}_{\Delta_2}}$
will carry a $b$-toric manifold structure.
\end{theorem}

\begin{proof}
The same orientation gluing
$\leftidx{_{\rm{cut}}}{X}_{\Delta_1}\cup\leftidx{^{\rm{cut}}}{X}_{\Delta_2}$
is easy to prove possessing a toric structure. In fact, because of
the local model proved in the above part, it is plain to see that
the symplectic form on
$\leftidx{_{\rm{cut}}}{X}_{\Delta_1}\cup\leftidx{^{\rm{cut}}}{X}_{\Delta_2}$
should be chosen as
\begin{equation}
\omega=\left\{
\begin{aligned}
\omega_{\Delta_1}&,&\text{when the point is in
$\leftidx{_{\rm{cut}}}{X}_{\Delta_1}$},\\
\omega_{\Delta_2}&,&\text{when the point is in
$\leftidx{^{\rm{cut}}}{X}_{\Delta_2}$},\\
\omega_{\Delta_F}+dh\wedge d\theta&,&\text{when the point is in the
cut loci}.
\end{aligned}
\right.
\end{equation}
Because of the local symplectomorphism above this is well-defined.
As for the Hamiltonian action, by Corollary 32, we directly define
the toric action as the usual action at points of
$\leftidx{_{\rm{cut}}}{X}_{\Delta_1}$,
$\leftidx{^{\rm{cut}}}{X}_{\Delta_2}$, and we define the action on
the cut loci as $\mathbb{T}^n=\mathbb{T}^{n-1}\times\mathbb{S}^1$
action on $X_{\Delta_F}\times\mathbb{S}^1$. This is obviously
well-defined, and the action is Hamiltonian with respect to the
moment map
\begin{align*}
\mu=\left\{
\begin{aligned}
\mu_{\Delta_1}&,&\text{when the point is in
$\leftidx{_{\rm{cut}}}{X}_{\Delta_1}$},\\
\mu_{\Delta_2}&,&\text{when the point is in
$\leftidx{_{\rm{cut}}}{X}_{\Delta_2}$},\\
(\mu_{\Delta_F},h)&,&\text{when the point is in the cut loci}.
\end{aligned}
\right.
\end{align*}
Here, of course, one should choose $\mu_{\Delta_i}$ suitably to make
the image pinched together at level $h$.

The difficult part is gluing through opposite orientation. In the
opposite gluing case, we will apply the partition as the Figure
\ref{fig:normal2}.

First, by the construction, the smooth structure is clear. For the
$\mathbb{T}^n$ action on it, we decompose it into
$\mathbb{T}^{n-1}\times\mathbb{S}^1$, as in Corollary 33. Let
$\mathbb{T}^{n-1}$ acts on
$\leftidx{_{\rm{cut}}}{X}_{\Delta_1}\cup\overline{\leftidx{^{\rm{cut}}}{X}_{\Delta_2}}$
by the usual action on $\leftidx{_{\rm{cut}}}{X}_{\Delta_1}$ and
$\leftidx{^{\rm{cut}}}{X}_{\Delta_2}$. However, for the
$\mathbb{S}^1$ action we let it act on
$\leftidx{_{\rm{cut}}}{X}_{\Delta_1}$ as usual, but act on
$\leftidx{^{\rm{cut}}}{X}_{\Delta_2}$ reversely. This action will be
compatible with the smooth structure. The question now is that the
moment map image will not be compatible together. Now we introduce
the partition function, as the following picture shows.

\begin{figure}[ht]
\includegraphics[width=6.3cm,height=6cm]{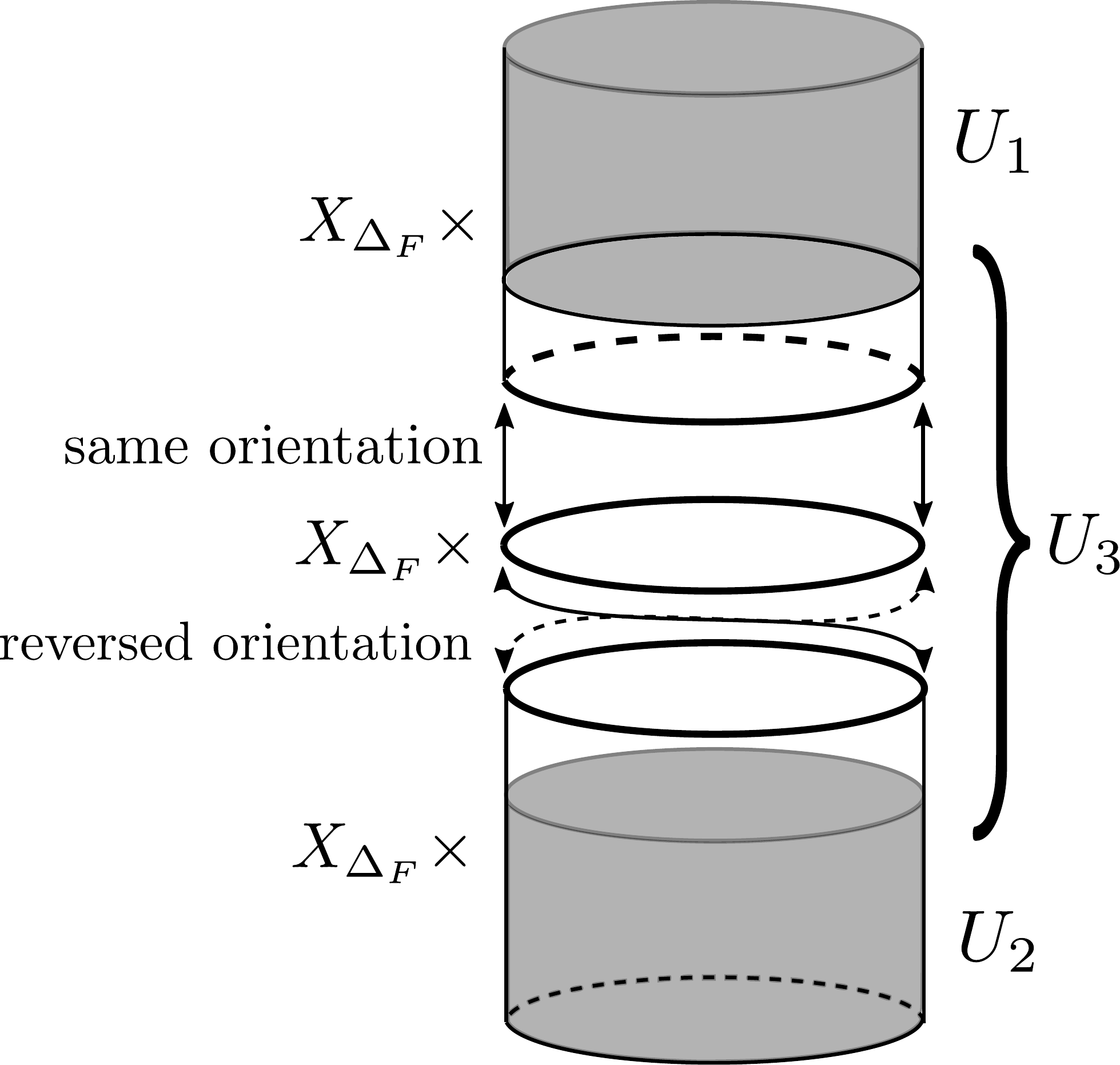}
\caption{The partition.} \label{fig:reversed}
\end{figure}
We use $U_3$ to denote cental part, and $U_1$ $U_2$ to denote the
shaded open set. Hence $U_i$ forms an open covering of
$\leftidx{_{\rm{cut}}}{X}_{\Delta_1}\cup\overline{\leftidx{^{\rm{cut}}}{X}_{\Delta_2}}$.
Define the partition functions $\rho_i$ compactly supported in $U_i$
and satisfy the conditions
\begin{itemize}
\item $\rho_1+\rho_2+\rho_3=1$.
\item $\rho_i\geq0$.
\end{itemize}
Now we define the $b$-symplectic form as
\begin{align*}
\omega=\rho_1\omega_{\Delta_1}+\rho_2\omega_{\Delta_2}+\rho_3(\omega_{\Delta_F}+c\frac{1}{h}dh\wedge
d\theta).
\end{align*}
Here $c$ is some nonzero constant. It is easy to verify this is a
$b$-symplectic form (in fact the part in the form that does not
involve $dh$ will form $\omega_{\Delta_F}$, and the part involving
$dh$ definitely is closed). Note that this $b$-symplectic form is
compatible with the orientation of
$\leftidx{_{\rm{cut}}}{X}_{\Delta_1}\cup\overline{\leftidx{^{\rm{cut}}}{X}_{\Delta_2}}$
no matter above or under the cut loci.

What left is to examine the action defined above is Hamiltonian. By
the decomposition of
$\mathbb{T}^n=\mathbb{T}^{n-1}\times\mathbb{S}^1$, we only need to
verify the existence of moment maps for both of them. Let
\begin{align*}
\iota_{\mathbb{T}^{n-1}}&:
\rm{Lie}\,(\mathbb{T}^{n-1})\rightarrow\rm{Lie}\,(\mathbb{T}^{n}),\\
\iota_{\mathbb{S}^{1}}&:
\rm{Lie}\,(\mathbb{S}^{1})\rightarrow\rm{Lie}\,(\mathbb{T}^{n}).
\end{align*}
For the $\mathbb{T}^{n-1}$ action, we can directly use the
definition $\iota_{\mathbb{T}^{n-1}}^*\circ\mu_{\Delta_i}$ same as
before, since the $\mathbb{T}^{n-1}$ action is not reversed. As for
the $\mathbb{S}^1$ action, in the part $U_1\backslash U_3$ we have
$\omega=\omega_{\Delta_1}$ and we simply define the moment map as
$\iota^*_{\mathbb{S}^1}\circ\mu_{\Delta_1}$. When it comes to
$U_3=X_{\Delta_F}\times\text{cylinder}$, by the local model above,
the action will be rotation along the cylinder. And to find a moment
map it suffices to find a function of $h$, $f(h)$ such that
$$df(h)=\rho_1dh+\rho_2dh+\rho_3c\frac{1}{h}dh.$$
Such a function obviously exists and by adding some constant, we can
make $f(h)$ and $\iota^*_{\mathbb{S}^1}\circ\mu_{\Delta_1}$ patch
together becomes a continuous function. Similarly, by adding some
constants to $-\iota^*_{\mathbb{S}^1}\circ\mu_{\Delta_2}$ we can
also make $f(h)$ and $-\iota^*_{\mathbb{S}^1}\circ\mu_{\Delta_2}$
path together to be a continuous function (note that there is a
negative sign, because the $\mathbb{S}^1$ action is reversed at this
part). We denote the function after patched together as above by
$f(h)$. Away from the cut loci it is obviously smooth. And around
the cut loci, it will take the form $f(h)=c\log|h|+g(h)$ for some
smooth function $g(h)$. Hence
$(\iota_{\mathbb{T}^{n-1}}^*\circ\mu_{\Delta_i},f(h))$ is a $b$-map
to the $b$-moment map codomain, satisfying the condition of a
$b$-moment map with respect to the toric action as above.

Till now we can conclude that the theorem holds.
\end{proof}

In fact, from the proof one can conclude more about the toric and
$b$-toric structure constructed above.
\begin{corollary}
For the toric structure constructed above, the Delzant polytope
correspond to it will be the union of $\Delta_1$ and $\Delta_2$
along $\Delta_F$, and the symplectic form of it is canonical. For
the $b$-toric structure, by the construction the $b$-symplectic form
is not canonical, and the corresponding $b$-Delzant polytope is
still the union of $\Delta_1$ and $\Delta_2$ along $\Delta_F$, with
negative infinity hyperplane at $F$, and the $b$-Delzant polytope
falls on the $b$-moment codomain coming from a weighted adjacent
graph with two vertices, one edge and the modular weight $-ct_n^*$.
\end{corollary}

Have already established the theorem above, now one can repeat this
kind of cutting and gluing process. Hence it is easy to see that all
$b$-toric manifolds can be got by cutting and gluing several toric
manifolds together.

\begin{example}
If we cut and glue two Hirzebruch surface $H_1$ together with
Delzant polytopes corresponding to the figure as follow, we will get
the corresponding toric and $b$-toric manifold shown in the picture \ref{fig:gluingexample}.
\begin{figure}[ht]
\includegraphics[width=5.8cm,height=6.5cm]{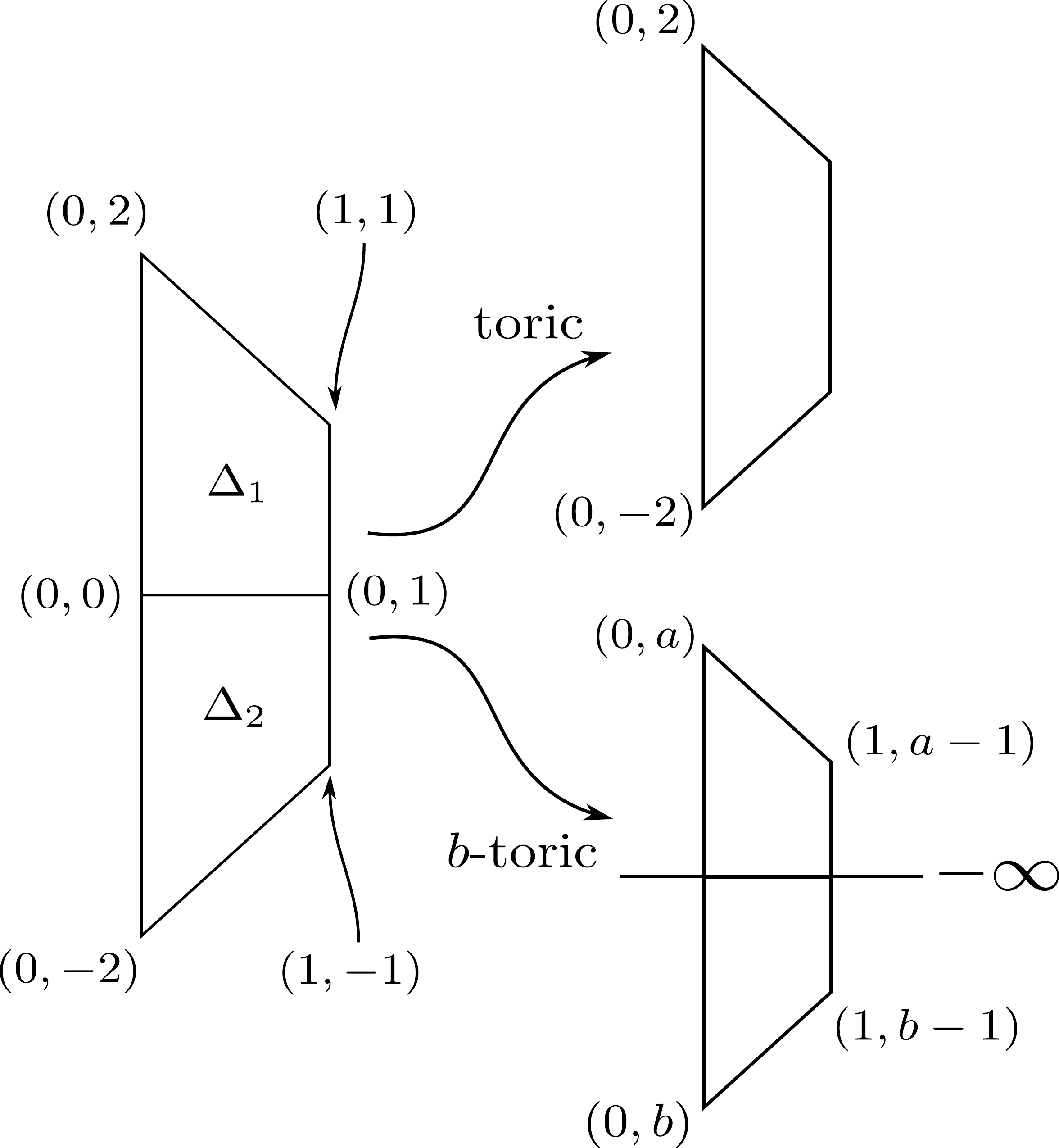}
\caption{Gluing Hirzebruch surfaces.} \label{fig:gluingexample}
\end{figure}

The $a$ and $b$ in the picture \ref{fig:gluingexample} will depend on our choice of
$b$-symplectic form.
\end{example}

\begin{remark}
In fact, if we examine the proof of Theorem 25 carefully, the above
cutting and gluing operation can be applied to parallel toric
manifold by cutting at a parallel hypersurface, then gluing along
the cut loci. And once again this can produce toric or $b$-toric
manifold, depending on the choice of orientation.
\end{remark}

\section{Decomposition to several toric manifolds}\label{section decomposition}

In this section we converse the construction in the above section.
Given a $b$-symplectic manifold we can decompose it into those
building blocks. But this decomposition, just as the construction,
is not canonical.

Given a $b$-toric manifold $(M,Z,\omega_M)$, by Proposition 26 in
\cite{delzant} we have local model for a $b$-toric manifold around
$Z$. That is to say, there will exists a small neighborhood $U$ of
$Z$ such that there exists a equivariant $b$-symplectomorphism from
$U$ to $X_{\Delta_F}\times\mathbb{S}^1\times\mathbb{R}$, with
symplectic form $\omega_{\Delta_F}+c\frac{1}{h}dh\wedge d\theta$.
Here $c$ is some suitable constant. Moreover the singular
hypersurface $Z$ is mapped to
$X_{\Delta_F}\times\mathbb{S}^1\times\{0\}$.

Consider a connected component in $M\backslash Z$, denoted by
$X_{\rm{sing}}$. Without loss of generality we can assume that there
is only one connected component in $Z$ that is adjacent to
$X_{\rm{sing}}$ and $X_{\rm{sing}}$ looks like
$X_{\Delta_F}\times\mathbb{S}^1\times\mathbb{R}^+$ in the local
model. By an abuse of notation we still denote this connected
component in $Z$ by $Z$. We will equip $X_{\rm{sing}}/Z$ a toric
structure (which has many different choices).\

Because of the local model we mentioned above, in $X\cap Z$, around
$Z$ we will have a local model for it and the quotient will be as in
the following picture.

\begin{figure}[ht]
\includegraphics[width=11cm,height=4cm]{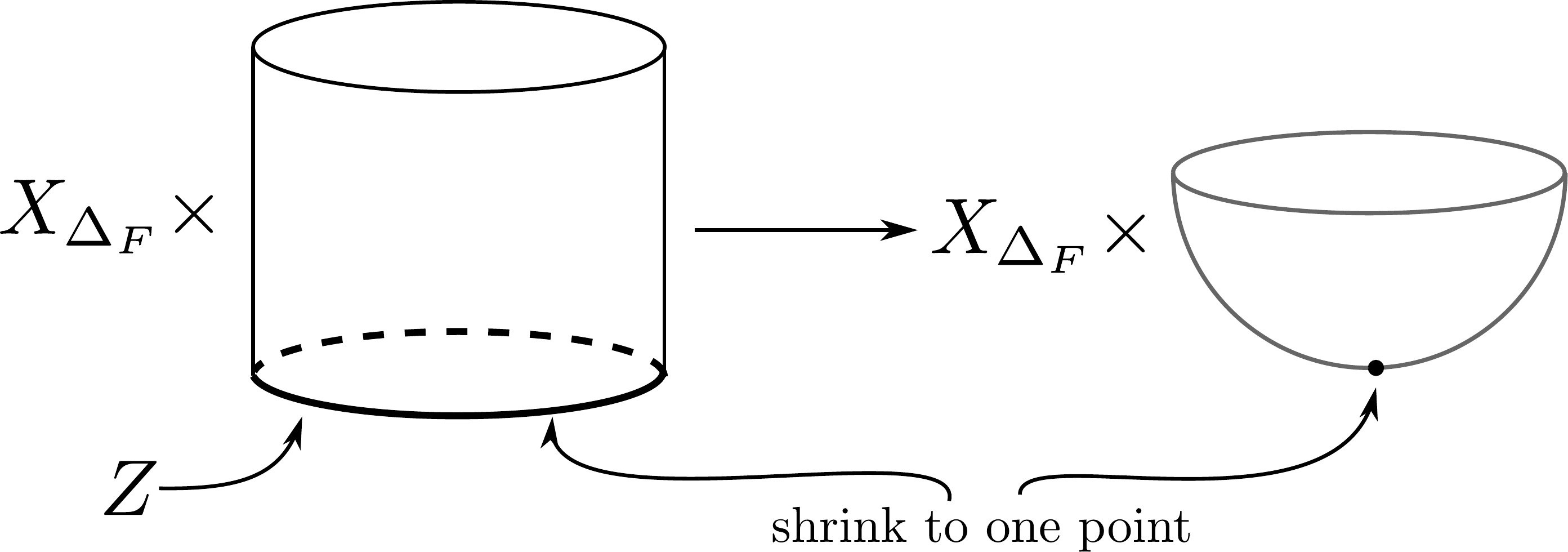}
\caption{The quotient operation.} \label{fig:shrink}
\end{figure}

From the local model, it is easy to give the quotient space
$X_{\rm{sing}}/Z$ a smooth structure, such that around the point
$[Z]$ in the quotient space $X_{\rm{sing}}/Z$ it looks like
$X_{\Delta_F}\times\text{part of a sphere}$. Now $X_{\rm{sing}}/Z$
is a closed smooth manifold.

\subsection{symplectic structures}

We need to attach the quotient space with a symplectic structure. It
turns out that there is not a canonical way to give a symplectic
structure and this is why we said the decomposition is not
canonical.

Still, we let $V$ be an open neighborhood of
$X_{\rm{sing}}\backslash U$, here we use $U$ to denote the local
model neighborhood mentioned above (i.e. $U$ symplectomorphic to the
model $X_{\Delta_F}\times\mathbb{S}^1\times\mathbb{R}$). Futhermore,
we assume $V\cap U$ takes the form
$X_{\Delta_F}\times\mathbb{S}^1\times(\delta,+\infty)$, for some $\delta$ sufficiently
small. In the quotient space we still use $U$ and $V$ to denote the
corresponding open sets. Now, $U$ and $V$ form an open cover of
$X_{\rm{sing}}/Z$, we define partition functions $\rho_U$ and
$\rho_V$ with respect to this open cover as before. Hence, one can
define the following symplectic form on $X_{\rm{sing}}/Z$
\begin{align*}
\omega=\rho_V\omega_M+\rho_U(\omega_{\Delta_F}+dh\wedge d\theta).
\end{align*}

It is quite clear that this is a well-defined symplectic form (the
reason is the same as our first partition operation). Also one can
see clearly why this form is not canonical.

\subsection{toric structures}

In this part we deal with the torus action. As a matter of fact the
original $\mathbb{T}^n$ action induces a $\mathbb{T}^n$ action on
the quotient space $X_{\rm{sing}}/Z$. Moreover, around the point
$[Z]\in X_{\rm{sing}}/Z$ this action can be split to
$\mathbb{T}^{n-1}$ action on $X_{\Delta_F}$ and $\mathbb{S}^1$
action on the part of a sphere. The only thing left to be checked is
Hamiltonian action. Indeed, after we split to $\mathbb{T}^{n-1}$ and
$\mathbb{S}^1$, similarly, we only need to check them separately.
For $\mathbb{T}^{n-1}$, since in $U$ the symplectic form $\omega$
becomes
\begin{align*}
\omega&=\rho_V(\omega_{\Delta_F}+c\frac{1}{h}dh\wedge
d\theta)+\rho_U(\omega_{\Delta_F}+dh\wedge d\theta),\\
&=\omega_{\Delta_F}+\rho_Vc\frac{1}{h}dh\wedge
d\theta+\rho_Udh\wedge d\theta,
\end{align*}
this part of the action is obviously Hamiltonian. For the second
part, still, in $(X_{\rm{sing}}/Z)\backslash U$ this is trivial and
in $U$, we can directly find a function $f(h)$ such that
$$df(h)=\rho_Vc\frac{1}{h}dh+\rho_Udh.$$
And by adding some suitable constant we can glue the moment map for
$V$ part and $f(h)$ together.

Till now we can conclude the space $X_{\rm{sing}}/Z$ does posses a
toric structure. And from the construction, it is easy to see that
the Delzant polytope will be the intersection of the part of the
$b$-Delzant polytope in $(\rm{Lie}\,\mathbb{T}^n)^*$ which
corresponds to $X_{\rm{sing}}$, with a upper half space defined by
some hyperplane parallel to the $t_1^*,\dots,t_{n-1}^*$ plane. And
the this hyperplane will be strict parallel to the Delzant polytope.

The following picture illustrates this situation.

\begin{figure}[ht]
\includegraphics[width=8cm,height=5.2cm]{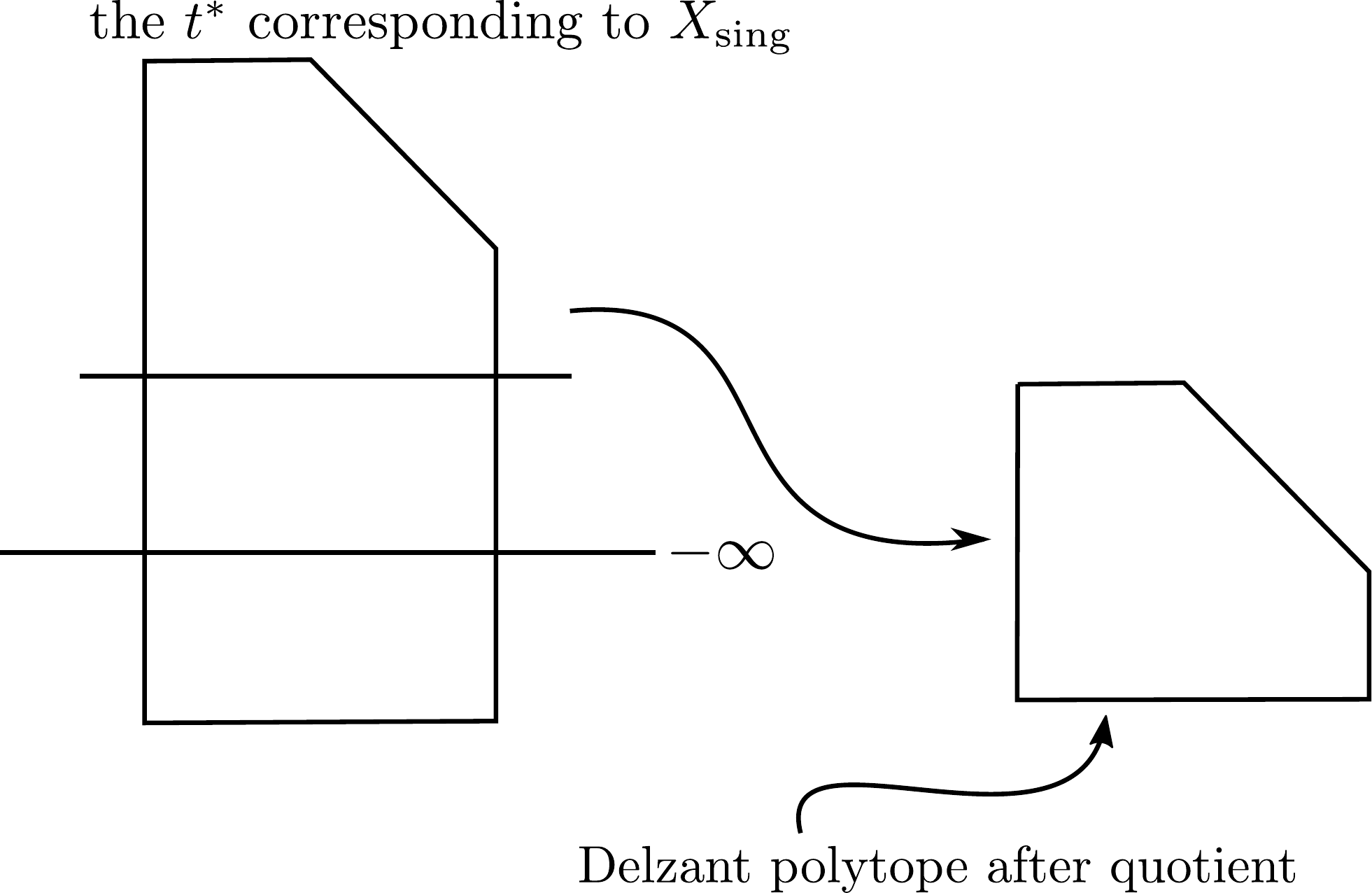}
\caption{Delzant polytope after decomposition.}
\label{fig:decomposition}
\end{figure}

Now, by the decomposition described above, one can decompose a
$b$-symplectic manifold piece by piece to get several toric
manifolds (the choice of symplectic structures on them are not
unique but the manifolds and the actions are unique). And it is easy
to see that from those decomposed manifolds, by the gluing technique
we can recover the $b$-symplectic manifold (still, the action and
$b$-manifold are unique but the $b$-symplectic form is not). From the process of decoposition and construction it is easy to see, if we decompose a $b$-toric manifold and then we use the partition functions used in the decomposition to glue those toric manifolds together to get a $b$-toric manifold, we will end up with the exactly same $b$-toric manifold, with same $b$-symplectic form, as we started. 

\begin{example}
For example consider $X_{\Delta_F}\times\mathbb{T}^2$, where
$\mathbb{T}^2$ will be equipped with $b$-toric structure with two
singular hypersurfaces, as in Example 25 in \cite{delzant} or
Example 8 as above. Decompose this will give us two
$X_{\Delta_F}\times\mathbb{S}^2$. And gluing them will return us
with $X_{\Delta_F}\times\mathbb{T}^2$.
\end{example}

%


\section{Applications}\label{section application}

In this section we will try to show that this kind of construction
and decomposition are useful.

\subsection{Homology of $b$-toric manifold}
In \cite{silva2}, it was shown that the homology of a toric manifold
can be computed perfectly by using Morse theory. Here, we will carry
this way to the $b$-toric case and compute the homology of a
$b$-toric manifold. It turns out that the homology of $b$-toric
manifold can be computed easily also.

First, it is plain to see that the case
$X_{\Delta_F}\times\mathbb{S}^2$ case is easy to compute. Because we
already know the homology of the usual toric manifold $\Delta_F$, by
K$\ddot{\rm{u}}$nneth formula one can get the homology of
$X_{\Delta_F}\times\mathbb{S}^2$.

So it boils down to consider the case that the adjacent graph does
no form a loop. Select a generic element
$X\in\rm{Lie}\,(\mathbb{T}^n)$, such that its components are
independent over $\mathbb{Q}$. Without loss of generality, we assume
that there is only one edge and two vertices in the adjacent graph
and hence the $b$-toric manifold will be $(M,Z,\omega)$, where $Z$
has only one component and $M\backslash Z$ will have only two
components. We also assume in the local model the $b$-moment map in
the local model takes the form $c\log|h|+g(h)$ with $c<0$, i.e. the
modular weight takes the form $-ct^*_n$ which is along the direction
of $-t^*_n\in(\rm{Lie}\,\mathbb{T}^n)^*$. The proof of the general
case will be almost the same. Denote those two components as $M_1$
and $M_2$. Corresponding to this adjacent graph, define the
following $1$ dimensional $b$-manifold as the image of our Morse
function.

\begin{itemize}
\item As a set, it is defined as
$\mathbb{R}_2\sqcup\{e\}\sqcup\mathbb{R}_1$.
\item For the smooth structure, we endow a smooth structure on it
such that the following map is smooth.
\begin{align*}
f(x)=\left\{
\begin{aligned}
\exp(\frac{x}{c})&,&\text{when $x\in\mathbb{R}_2$},\\
0&,&\text{when $x\in\{e\}$},\\
-\exp(\frac{x}{c})&,&\text{when $x\in\mathbb{R}_1$}.
\end{aligned}
\right.
\end{align*}
where $c$ depends on the modular weight.
\end{itemize}

Just as in Section 5 of \cite{btoric}, such a smooth structure
exists. This is simply identifying the $-\infty$ of $\mathbb{R}_2$
and the $-\infty$ of $\mathbb{R}_1$ together to form a point $e$.
This Morse function codomain will be denoted as $b_M$. We introduce
an order $(b_M,\lessdot)$ structure on it, by
\begin{itemize}
\item for $a$, $b\in\mathbb{R}_2$, $a\lessdot b$ if $a>b$.
\item for $a\in\mathbb{R}_2\cup\{e\}$, $b\in\mathbb{R}_1$, $a\lessdot
b$.
\item for $a\in\mathbb{R}_2$, $b\in\mathbb{R}_1\cup\{e\}$, $a\lessdot
b$.
\item for $a$, $b\in\mathbb{R}_2$, $a\lessdot b$ if $a<b$.
\end{itemize}

The Morse function will be defined to be
$\mu^X=\langle\mu,X\rangle$, with $\mu^X$ takes value in
$\mathbb{R}_1$ when the point is in $M_1$, $\mu^X$ takes value in
$\mathbb{R}_2$ when the point is in $M_2$ and $\mu^X$ takes value in
$\{e\}$ when the point is in the singular hypersurface $Z$. This
function is clearly a $b$-map and is smooth. We will consider the
two components $M_i$ separately. Note that they can be viewed as
open toric manifolds from our construction in Section \ref{section
construction}. Because of the fact that the action is Hamiltonian
the equality $d\mu_X=-\iota_{X^{\#}}\omega$ holds, which implies
that the critical points of this function can only appear in the
fixed point set of the action in $M_i$. By our decomposition and
construction around a fixed point $p$ the symplectic form and action
are the same with the usual toric manifold. So as in \cite{silva2}
there is a chart such that the symplectic form and the moment map
take the form
\begin{align*}
\omega=\sum_{k=1}^ndx_k\wedge dy_k.
\end{align*}
\begin{align*}
\mu=\mu(p)-\frac{1}{2}\sum_{k=1}^n\lambda^{(k)}(x_k^2+y_k^2).
\end{align*}
Here the $-\lambda_k$'s are precisely the edge vectors of the fixed
point $p$ in the $b$-Delzant polytope (this can be seen easily once
we decompose the $b$-toric manifold to toric manifolds). Hence the
index at a fixed point $p$ can be read from the $b$-Delzant
polytope. That is to say, taking twice the number of edges which are
pointing up with respect to $X$ (i.e.
$\langle-\lambda^{(k)},X\rangle>0$). This index will be denoted as
$\rm{ind}_p$

\begin{figure}[ht]
\includegraphics[height=5.1cm]{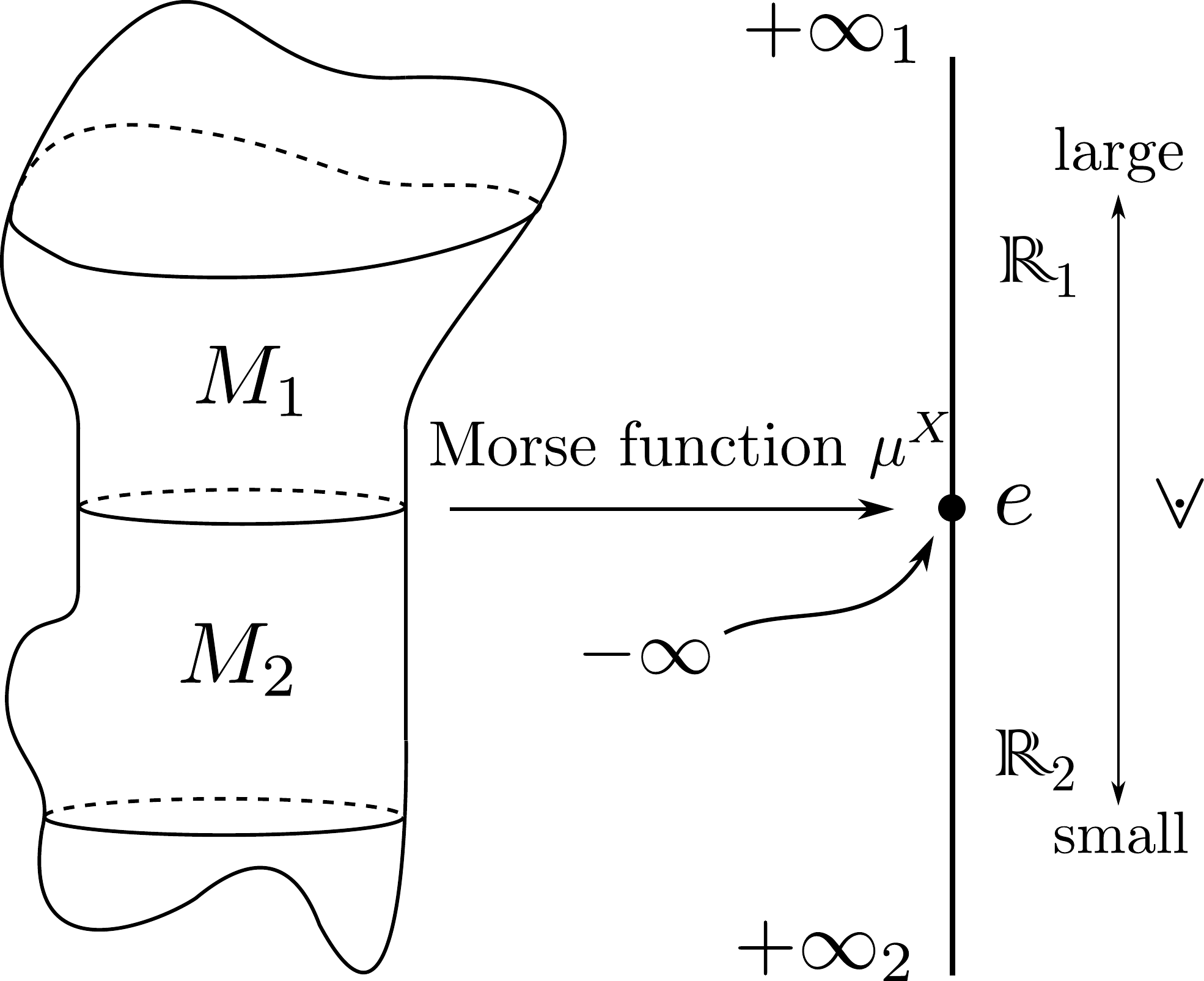}
\caption{The Morse function.} \label{fig:morse}
\end{figure}

So now let $p_1,\ldots,p_s$ denote the fixed points in $M_2$ and
$q_1,\ldots,q_t$ denote the fixed points in $M_1$. Moreover, assume
they are arranged such that
$\mu^{X}(p_1)\lessdot\ldots\lessdot\mu^{X}(p_s)\in\mathbb{R}_2$ and
$\mu^{X}(q_1)\lessdot\ldots\lessdot\mu^{X}(q_t)\in\mathbb{R}_1$
(from the $b$-Delzant polytope it is easy to see that it is
impossible to have $\mu^{X}(p_i)=\mu^{X}(p_j)$, because our way of
selecting $X$). If we use $+\infty_2$ to denote the positive
infinity of $\mathbb{R}_2$, and we use interval to denote interval
under the order structure $(b_M,\lessdot)$ one will have,
\begin{theorem}
\

\begin{enumerate}[(1)]

\item Given an interval
$[a,b]\subset b_M$ (it is possible for $a$ and $b$ are in different
$\mathbb{R}_i$), such that $[a,b]$ does not contain any critical
value of $\mu^X$, then $(\mu^X)^{-1}(+\infty_2,b]$ has the homotopy
type of $(\mu^X)^{-1}(+\infty_2,a]$.

\item For $\epsilon$ small enough,
$(\mu^X)^{-1}(+\infty_2,\mu^{X}(p_i)-\epsilon]$ will have the
homotopy type of $(\mu^X)^{-1}(+\infty_2,\mu^{X}(p_i)+\epsilon]$
with a cell of dimension $\rm{ind}_{p_i}$.

\item For $\epsilon$ small enough,
$(\mu^X)^{-1}(+\infty_2,\mu^{X}(q_j)+\epsilon]$ will have the
homotopy type of $(\mu^X)^{-1}(+\infty_2,\mu^{X}(q_j)-\epsilon]$
with a cell of dimension $\rm{ind}_{q_j}$.
\end{enumerate}
\end{theorem}
In fact, in the above argument we have already proved (2) and (3).
The only thing left is the case of (1) that the interval contains
point $e$. But by our local model in the construction, around the
singular hypersurface it is symplectomorphic to
$X_{\Delta_F}\times\mathbb{R}\times\mathbb{S}^1$, and the upper half
$M_1$ in this local model is
$X_{\Delta_F}\times(0,+\infty)\times\mathbb{S}^1$. It will be
similar for the lower half. If we write the generic vector $X$ as
$(X',x_n)$, the Morse function will be
$$\langle\mu_{\Delta_F},X'\rangle+c\log|h|x_n.$$
Now in the local model, the case is shown as in the picture as
follows.

\begin{figure}[ht]
\includegraphics[height=5.1cm]{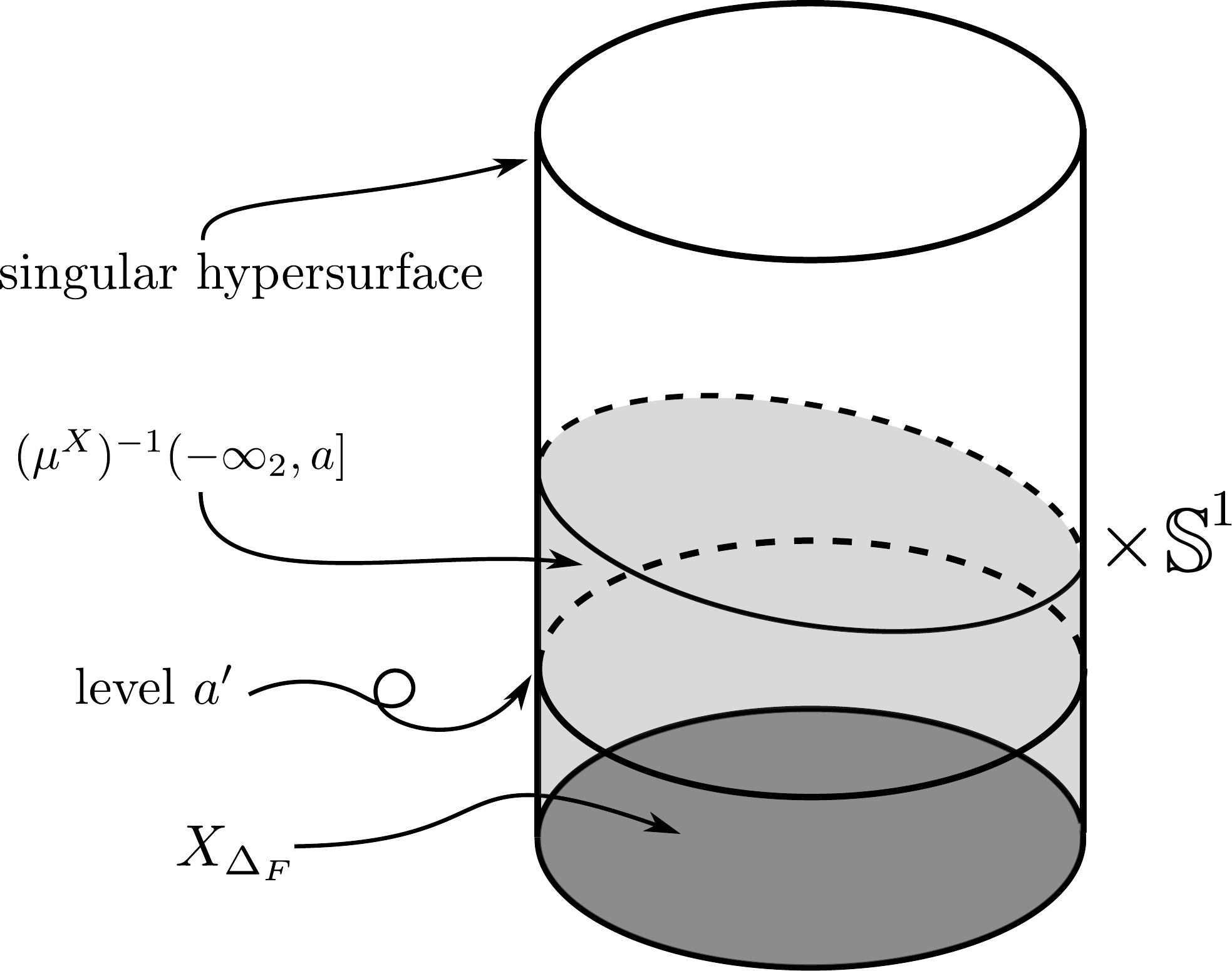}
\caption{$(\mu^X)^{-1}(-\infty_2,a]$ in local model.}
\label{fig:level}
\end{figure}

Without loss of generality, we can assume $|a|$ and $|b|$ are
sufficiently large. The set $(\mu^X)^{-1}(+\infty_2,a]$ will be
equivalent to say the points satisfy
$$\langle\mu_{\Delta_F},X'\rangle+c\log|h|x_n\geq a,$$
in the local model. Notice that $a$ is large and
$\langle\mu_{\Delta_F},X'\rangle$ is bounded because $X_{\Delta_F}$
is a compact manifold. Hence, there must exist an $a'<0$ such that
$$\inf\langle\mu_{\Delta_F},X'\rangle+c\log|a'|x_n\geq a,$$
in the local model, which implies
$X_{\Delta_F}\times(-\infty,a']\times\mathbb{S}^1\subset(\mu^X)^{-1}(+\infty_2,a]$.
It follows that we can define a deformation retract by shrinking
$(\mu^X)^{-1}(+\infty_2,a]$ to
$X_{\Delta_F}\times(-\infty,a']\times\mathbb{S}^1$ in this local
model. Indeed, if $(x,t,\theta)\in(\mu^X)^{-1}(+\infty_2,a]$, then
for all $t'<t$ one has $(x,t',\theta)\in(\mu^X)^{-1}(+\infty_2,a]$.
Hence we can define the deformation retraction
\begin{align*}
F_s(x,t,\theta)=\left(x,st+(1-s)a',\theta\right).
\end{align*}
for $(x,t,\theta)\in(\mu^X)^{-1}(+\infty_2,a]$.

One can do the exactly same thing to $(\mu^X)^{-1}(+\infty_2,b]$ the
upper half, and get some $b'$ such that $(\mu^X)^{-1}(+\infty_2,b]$
can be deformation retracted to
$X_{\Delta_F}\times(-\infty,b']\times\mathbb{S}^1$ in this local
model. And finally
$X_{\Delta_F}\times(-\infty,b']\times\mathbb{S}^1$ obviously can be
deformation retracted to
$X_{\Delta_F}\times(-\infty,a']\times\mathbb{S}^1$ in the local
model. Hence the theorem holds.

Notice that the index of a fixed point is always even as we said
before, hence it is plain to see that as a $CW$-complex, it consists
of even cells and by considering cellular homology we have

\begin{theorem}
Given a generic vector $X\in\mathbb{R}^n$ then for the $b$-toric
manifold $(M,Z,\omega)$, whose adjacent graph does not form a loop
\begin{align*}
\rm{dim}\,H_{2k}(M,\mathbb{Z})&=\#\{\text{fixed point }p|\ \rm{ind}_p= k\}\\
&=\text{number of vertices of the $b$-Delzant polytope such that}\\
&\ \ \ \ \ \ \ \text{there are $k$ edges pointing up with respect to
$X$}.
\end{align*}
\end{theorem}
Hence by the construction and decomposition, we can read the
homology of a $b$-toric manifold from the $b$-Delzant polytope.

\end{document}